\documentclass[12pt,reqno]{amsart}

\usepackage{amsmath,amssymb,ifthen}
\usepackage{fullpage}
\usepackage{graphicx,psfrag,subcaption}
\usepackage{color}
\usepackage{moreverb}
\usepackage{amsthm}
\usepackage{mathrsfs}
\usepackage{hhline}
\usepackage{bbm}
\usepackage[english]{babel}
\usepackage{tikz}
\usetikzlibrary{calc,math}
\usepackage{pgfplots}
\usepackage[utf8]{inputenc}
\usepackage[T1]{fontenc}
\usepackage{amsfonts}
\usepackage{enumerate}
\usepackage{hyperref}

\DeclareMathOperator*{\argmax}{argmax}

\renewcommand{\d}{\,{\rm d}} 
\DeclareMathOperator{\diam}{diam} 
 
\DeclareMathOperator{\supp}{supp}

\newcommand{\set}[2]{\big\{#1\,:\,#2\big\}}

\newcommand{\dual}[3][]{#1\langle#2\,,\,#3#1\rangle}

\newcommand{\norm}[3][]{#1\|#2#1\|_{#3}}

\newcommand\K{\mathbb{K}}
\newcommand\N{\mathbb{N}}
\newcommand\R{\mathbb{R}}

\newcommand\Z{\mathbb {Z}}

\newcommand\OO{{\mathcal O}}

\numberwithin{equation}{section}
\numberwithin{figure}{section}
\newtheorem{theorem}{Theorem}[section]
\newtheorem{proposition}[theorem]{Proposition}
\newtheorem{lemma}[theorem]{Lemma}

\newtheorem{algorithm}[theorem]{Algorithm}

\newtheorem{remark}[theorem]{Remark}

\newcommand*\patchAmsMathEnvironmentForLineno[1]{%
  \expandafter\let\csname old#1\expandafter\endcsname\csname #1\endcsname
  \expandafter\let\csname oldend#1\expandafter\endcsname\csname end#1\endcsname
  \renewenvironment{#1}%
     {\linenomath\csname old#1\endcsname}%
     {\csname oldend#1\endcsname\endlinenomath}}%
\newcommand*\patchBothAmsMathEnvironmentsForLineno[1]{%
  \patchAmsMathEnvironmentForLineno{#1}%
  \patchAmsMathEnvironmentForLineno{#1*}}%
\AtBeginDocument{%
\patchBothAmsMathEnvironmentsForLineno{equation}%
\patchBothAmsMathEnvironmentsForLineno{align}%
\patchBothAmsMathEnvironmentsForLineno{flalign}%
\patchBothAmsMathEnvironmentsForLineno{alignat}%
\patchBothAmsMathEnvironmentsForLineno{gather}%
\patchBothAmsMathEnvironmentsForLineno{multline}%
}
\usepackage[mathlines]{lineno}

\usepackage{bbm} 
\usepackage{caption}
\usepackage{subcaption}
\usepackage{soul}
\usepackage{pgfplots}
\pgfplotsset{compat=1.18}
\usepgfplotslibrary{groupplots}
\usepackage{float}

\renewcommand{\d}{\,{\rm d}}

\newcommand{\n}{{\bf n}}
\newcommand{\x}{{\bf x}}
\newcommand{\y}{{\bf y}}
\newcommand{\z}{{\bf z}}

\renewcommand{\H}{\widetilde{H}}
\renewcommand{\K}{\mathscr{K}}
\newcommand{\M}{\mathscr{M}}
\newcommand{\NN}{\mathcal{N}}
\newcommand{\PP}{\mathcal{P}}
\newcommand{\V}{\mathscr{V}}
\newcommand{\W}{\mathscr{W}}

\title{Adaptive space-time BEM\\
for the heat equation\\ 
with Neumann boundary conditions}

\author{Gregor~Gantner}
\address{University of Twente, Department of Applied Mathematics, P.O. Box 217, 7500 AE Enschede, The Netherlands}
\email{gregor.gantner@utwente.nl}

\author{Helmut Harbrecht}
\address{Departement f\"ur Mathematik und Informatik, Universit\"at Basel, Spiegelgasse 1, 4051 Basel, Switzerland}
\email{helmut.harbrecht@unibas.ch}

\author{Franz Nowakowsky}
\address{University of Twente, Department of Applied Mathematics, P.O. Box 217, 7500 AE Enschede, The Netherlands}
\email{franz.nowakowsky@utwente.nl}

\thanks{GG and FN acknowledge funding by the Deutsche Forschungsgemeinschaft (DFG, German Research Foundation) under Germany’s Excellence Strategy - EXC-2047/1 - 390685813 and by the Nederlandse Organisatie voor Wetenschappelijk Onderzoek (NWO, Dutch Research Council) through the Vidi project Optimal adaptive space-time boundary and finite element methods (OASTMethods) with file number VI.Vidi.243.148 under the grant https://doi.org/10.61686/LTBER75172. We further acknowledge the access to the Marvin cluster of the University of Bonn to perform the numerical experiments.}

\begin{document}

\begin{abstract}
We consider the space-time boundary element method (BEM) for the heat equation with prescribed initial and Neumann data.
We propose a weighted-residual {\sl a posteriori} error estimator that is an upper bound for the unknown BEM error.
The possibly locally refined meshes are assumed to be parabolically scaled prismatic, i.e., their elements are tensor-products $J\times K$ of elements in time $J$ and space $K$ with $|J| \eqsim \diam(K)^2$. 
In the considered numerical experiments on two-dimensional domains in space, an adaptive algorithm steered by the derived estimator yields significantly faster convergence compared to uniform refinement, achieving near-optimal rates even in the presence of strong singularities.
\end{abstract}

\date{\today}
\maketitle

 \section{Introduction}

In the last years, there has been a growing interest in simultaneous space-time boundary element methods (BEM) for the heat equation~\cite{cs13,mst14,mst15,ht18,cr19,dns19,dzomk19,tausch19,zwom21,wmoz22, wo22, wo23, fgk25, aktw25}.
In contrast to the differential-operator-based variational formulation on the space-time cylinder, the variational formulation corresponding to space-time BEM is coercive~\cite{an87,costabel90} so that the discretized version always has a unique solution  regardless of the chosen trial space, which is even quasi-optimal in the natural energy norm.
Moreover, it is naturally applicable on unbounded domains and only requires a mesh of the lateral boundary of the space-time cylinder (as well as a mesh of the spatial domain in case of nonhomogeneous initial data) resulting in a dimension reduction.
The potential disadvantage that discretizations lead to dense matrices due to the nonlocality of the boundary integral operators has been tackled, e.g., in \cite{bn01a,bn01b,tausch07,mst14,mst15,ht18,wmoz22,aktw25} via wavelets, the fast multipole method, and $\mathcal{H}$-matrices.

Two often mentioned advantages of simultaneous space-time methods are their potential for massive parallelization as well as their potential for fully adaptive refinement to resolve singularities local in both space and time.
While the first advantage has been investigated in, e.g., \cite{dzomk19,zwom21}, the latter requires suitable {\slshape a posteriori} computable error estimators, which are so far only available for the heat equation with Dirichlet boundary conditions \cite{gv21,gv22}. 
In the latter work, the Faermann error estimator \cite{faermann00, faermann02} for the Laplace problem is generalized.

In the present manuscript, we extend the weighted-residual estimators~\cite{cs95,carstensen97,cmps04} for the Laplace problem to the heat equation with Neumann boundary conditions:
Let $\Omega\subset\R^{d}$, $d=2,3$, be a Lipschitz domain with connected boundary $\Gamma:=\partial\Omega$ and $T>0$ a given end time point with corresponding time interval $I:=(0,T)$. While extensions to piecewise smooth curved boundaries are possible, for the ease of presentation, we assume that $\Omega$ is polygonal for $d=2$ and polyhedral for $d=3$.
We abbreviate the space-time cylinder $Q:=I\times\Omega$ with lateral boundary $\Sigma:=I\times\Gamma$ and corresponding outer normal vector $\n\in\R^{d}$.
With the heat kernel
\begin{align*}
	G(t,\x) := \begin{cases} \frac{1}{(4\pi t)^{d/2}} \, e^{-\frac{|\x|^2}{4t}} \quad &\text{for }(t,\x)\in (0,\infty)\times\R^{d},\\
	0 \quad &\text{else},
 	\end{cases}
\end{align*}
and a given function $f:\Sigma\to\R$, we consider the boundary integral equation
\begin{align}\label{eq:hypersing operator}
	(\mathscr{W} u)(t,\x) := -\partial_{\n(\x)} \int_\Sigma \partial_{\n(\y)}G(t-s,\x-\y) u(s,\y) \d \y \d s = f(t,\x) \quad\text{for }(t,\x)\in \Sigma.
\end{align}
Here, $\mathscr{W}$ is the hyper-singular operator. Note that the integral kernel is given as
\begin{equation}\label{eq:partialnyG}
  \partial_{\n(\y)}G(t-s,\x-\y)=(\x-\y)\cdot \n(\y)\frac{1}{2(t-s)}G(t-s,\x-\y).
\end{equation}
For given initial condition $U_0:\Omega\to\R$ and Neumann datum $\phi:\Sigma\to\R$, such equations arise from the heat equation
\begin{align}\label{eq:interior}
	\begin{array}{rcll}
	\partial_t U - \Delta_\x U & = & 0 & \text{ on } Q,\\
	\partial_\n U & = & \phi & \text{ on }\Sigma,\\
	U(0,\cdot) & = & U_0 & \text{ on }\Omega;
	\end{array}
\end{align}
see Section \ref{sec:integral equations}.

Let $\PP_h$ be a mesh of the space-time boundary $\Sigma$ consisting of prismatic elements $J\times K$ with intervals  $J\subseteq \overline I$ and simplices $K\subseteq \Gamma$, and $\widetilde S^{p_t,p_\x}(\PP_h)$ the associated space of continuous piecewise polynomials of some fixed degree in time and space that vanish at $t=0$.
Further, let $u_h\in\widetilde S^{p_t,p_\x}(\PP_h)$ be the Galerkin approximation of $u$ with respect to the coercive bilinear form induced by $\W$.
As $\W$ is an isomorphism from the anisotropic Sobolev space $H^{1/2,1/4}(\Sigma):= L^2 (I; H^{1/2}(\Gamma)) \cap H^{1/4}(I; L^2(\Gamma))$  to its dual space $H^{-1/2,-1/4}(\Sigma):=H^{1/2,1/4}(\Sigma)'$, the discretization error $\norm{u-u_h}{H^{1/2,1/4}(\Sigma)}$ is equivalent to the norm of the residual $\norm{f-\mathscr{W}u_h}{H^{-1/2,-1/4}(\Sigma)}$.
Under certain regularity assumptions on the mesh and the \emph{local} parabolic scaling $|J|\eqsim\diam(K)^2$ for all $J\times K\in\PP_h$, we show that
\begin{equation}\label{eq:intreliability}
	\norm{u-u_h}{H^{1/2,1/4}(\Sigma)}^2 \eqsim \norm{f-\mathscr{W}u_h}{H^{-1/2,-1/4}(\Sigma)}^2
	\lesssim \sum_{J\times K\in\PP_h} \diam(K)  \norm{f-\W u_h}{L^2(J\times K)}^2.
\end{equation}
Assuming $f\in L^2(\Sigma)$, the right-hand side only makes sense if $\mathscr{W}u_h\in L^2(\Sigma)$, which follows from the regularity $\mathscr{W}:\H^{1,1/2}(\Sigma) \to L^2(\Sigma)$ in conjunction with the inclusion $\widetilde S^{p_t,p_\x}(\PP_h)\subset \H^{1,1/2}(\Sigma):=\set{v \in L^2(I;H^1(\Gamma)) \cap H^{1/2}(I;L^2(\Gamma))}{v(0,\cdot) = 0}$. 
 Such a regularity result is valid for piecewise smooth (including in particular polygonal and polyhedral) domains $\Omega$; see  \cite[Section 4]{costabel90}. 

Building on the derived {\sl a posteriori} computable error estimator \eqref{eq:intreliability}, we further propose an adaptive algorithm and numerically investigate it.
In all considered experiments (for $d=2$),  reliability of the estimator \eqref{eq:intreliability} is empirically confirmed and also efficiency, i.e., the converse inequality, is observed. 
Compared to uniform refinement, the adaptive algorithm yields significantly faster convergence, achieving near-optimal rates with respect to the number of degrees of freedom, even in the presence of strong singularities. 
We provide full details on the numerical implementation of the proposed algorithm, including the computation of the Galerkin system and the error estimator. 
In particular, for arbitrary polynomial degree $p_t$ in time, we give explicit formulas for the time integrals in the boundary integral operators applied to and tested with piecewise polynomials of that degree, which are so far only available for $p_t = 0$; see, e.g., \cite{costabel90,reinarz15,zwom21}.

\subsection*{Outline}

The remainder of this work is organized as follows: Section 2 summarizes the general principles of the space-time boundary element method for the heat equation, including prismatic boundary meshes and piecewise polynomial ansatz spaces along with a basis for the lowest-order case $p_t = p_x = 1$. 
Section 3 establishes Poincaré-type inequalities for anisotropic Sobolev spaces (Section 3.1) and constructs a Clément-type interpolation operator (Section 3.2).
The latter is used to  derive our main result, reliability of a weighted-residual estimator (Theorem 3.6).  
Finally, Section 4 introduces an adaptive algorithm which is based on the derived error estimator. 
This algorithm is subsequently applied for $d=2$ to several concrete examples with typical singularities in space and time. 
The stable implementation is discussed in Appendix A.

\section{Preliminaries}\label{sec:preliminaries}

\subsection{General notation}
Throughout and without any ambiguity, $|\cdot|$ denotes the absolute value of scalars, the Euclidean norm of vectors in $\R^n$, or the the measure of a set in $\R^n$, e.g., the length of an interval or the area of a surface in $\R^3$.
We write $A\lesssim B$ to abbreviate $A\le CB$ with some generic constant $C>0$, which is clear from the context.
Moreover, $A\eqsim B$ abbreviates $A\lesssim B\lesssim A$.

\subsection{Anisotropic Sobolev spaces}\label{sec:anisotropic spaces}
For a measurable $d$-dimensional $\omega \subseteq \Omega$ or $(d-1)$-dimensional $\omega \subseteq \Gamma$, and $\mu\in(0,1]$, we first recall the Sobolev space
\begin{align*}
	H^{\mu}(\omega):=\set{v\in L^2(\omega)}{\norm{v}{H^{\mu}(\omega)} < \infty}
\end{align*}
associated with the Sobolev--Slobodeckij norm
\begin{align*}
	\norm{v}{H^{\mu}(\omega)}^2
	:= \norm{v}{L^2(\omega)}^2 + |v|_{H^{\mu}(\omega)}^2,
	\quad  |v|_{H^{\mu}(\omega)}^2
	:=\begin{cases}
	\int_\omega\int_\omega \frac{|v(\x)-v(\y)|^2}{|\x-\y|^{{\rm dim}(\omega)+2\mu}}\d\y\d\x &\text{ if }\mu\in(0,1),
	\\
	\norm{\nabla_\omega v}{L^2(\omega)}^2 \quad&\text{ if }\mu=1,
	\end{cases}
\end{align*}
where ${\rm dim}(\omega)$ denotes the dimension of $\omega$, i.e.,  $d$ or $d-1$, and $\nabla_\omega$ denotes the (weak) gradient on $\omega$, i.e., the standard gradient or the surface gradient.

Moreover, we define for any subinterval $J\subseteq\overline I$, $\nu\in(0,1]$, and any Banach space $X$,
\begin{align*}
	H^\nu(J;X) := \set{v\in L^2(J;X)}{\norm{v}{H^{\nu}(J;X)} < \infty}
\end{align*}
associated with the norm
\begin{align*}
	\norm{v}{H^{\nu}(J;X)}^2
	:= \norm{v}{L^2(J;X)}^2 + |v|_{H^{\nu}(J;X)}^2,
	\quad  |v|_{H^{\nu}(J;X)}^2
	:=\begin{cases}
	\int_J\int_J \frac{\norm{v(t)-v(s)}{X}^2}{|t-s|^{1+2\nu}}\d s\d t &\text{ if }\nu\in(0,1), 
	\\
	\norm{\partial_t v}{L^2(J;X)}^2 &\text{ if }\nu=1,
	\end{cases}
\end{align*}
where $\partial_t$ denotes the (weak) time derivative.
If $X=\R$, we simply write $H^\nu(J)$, $\norm{v}{H^\nu(J)}$, and $|v|_{H^\nu(J)}$.
We recall the anisotropic Sobolev space
\begin{align*}
	H^{\mu,\nu}(J\times\omega) := L^2(J;H^{\mu}(\omega)) \cap H^{\nu}(J;L^2(\omega))
\end{align*}
with corresponding norm
\begin{align*}
	\norm{v}{H^{\mu,\nu}(J\times\omega)}^2
	:= \norm{v}{L^2(J;H^{\mu}(\omega))}^2 + \norm{v}{H^{\nu}(J;L^2(\omega))}^2
	\quad\text{for all } v\in H^{\mu,\nu}(J\times\omega).
\end{align*}
We will sometimes use the abbreviation
\begin{align*}
	|v|_{L^2(J;H^{\mu}(\omega))}^2 := \int_J |v(t,\cdot)|_{H^\mu(\omega)}^2 \d t
	\quad\text{for all } v\in L^2(J;H^\mu(\omega)).
\end{align*}

For $\omega\in\{\Omega,\Gamma\}$, we denote by $H^{-\mu,-\nu}(I\times\omega)$ the dual space of $H^{\mu,\nu}(I\times\omega)$ with duality pairing $\dual{\cdot}{\cdot}_{I\times\omega}$.
We interpret $L^2(I\times\omega)$ as subspace of $H^{-\mu,-\nu}(I\times\omega)$ via
\begin{align*}
	\dual{v}{\psi}_{I\times\omega} := \int_I\int_\omega v(t,\x) \psi(t,\x) \d\x \d t
	\quad\text{for all }v\in H^{\mu,\nu}(I\times\omega) \text{ and }\psi\in L^2(I\times\omega).
\end{align*}
We further recall the space
\begin{align*}
	\H^{\mu,\nu}(I\times\omega) := \set{v\in H^{\mu,\nu}(I\times\omega)}{v(0,\cdot)=0},
\end{align*}
which coincides with $H^{\mu,\nu}(I\times\omega)$ for $\nu<1/2$, and 
\begin{align*}
	H^{1,1/2}{(Q;\partial_t-\Delta_\x)}:=\set{v\in H^{1,1/2}(Q)}{\partial_tv-\Delta_\x v\in L^2(Q)},
\end{align*}
which is equipped with the canonical graph norm. For more details on the defined anisotropic Sobolev spaces, we refer, e.g., to \cite{costabel90}.

\subsection{Boundary integral equations}\label{sec:integral equations}
For given initial condition $U_0:\Omega\to\R$ and Neumann datum $\phi:\Sigma\to\R$, we consider the
heat equation \eqref{eq:interior}.
It is well-known that for $U_0\in L^2(\Omega)$ and $\phi\in H^{-1/2,-1/4}(\Sigma)$, the heat equation~\eqref{eq:interior} admits a unique solution
$u\in H^{1,1/2}(Q)$.
With the lateral trace $u:=U|_\Sigma \in H^{1/2,1/4}(\Sigma)$, $U$ satisfies the representation formula
\begin{align}\label{eq:representation}
	U = \widetilde \M_0 U_0 + \widetilde \V\phi  - \widetilde \K u,
\end{align}
where
\begin{align}\label{eq:initial potential}
	(\widetilde \M_0 U_0)(t,\x)&:=\int_\Omega G(t,\x-\y) U_0(\y) \d\y \quad\text{for all }(t,\x)\in Q
\intertext{denotes the initial potential,}
	\label{eq:single potential}
	(\widetilde \V \phi)(t,\x)&:=\int_\Sigma G(t-s,\x-\y) \phi(s, \y) \d\y \d s \quad\text{for all }(t,\x)\in Q
\intertext{denotes the single-layer potential, and }
	\label{eq:double potential}
	(\widetilde \K u)(t,\x)&:=\int_\Sigma \partial_{\n(\y)}G(t-s,\x-\y) u(s, \y) \d\y \d s \quad\text{for all }(t,\x)\in Q
\end{align}
denotes the double-layer potential.
These linear operators satisfy the mapping properties $\widetilde \M_0:L^2(\Omega)\to H^{1,1/2}(Q;\partial_t-\Delta_\x)$, $\widetilde \V:H^{-1/2,-1/4}(\Sigma)\to H^{1,1/2}(Q;\partial_t-\Delta_\x)$, and $\widetilde \K:H^{1/2,1/4}(\Sigma)\to H^{1,1/2}(Q;\partial_t-\Delta_\x)$.
We denote the normal derivative $\partial_\n(\cdot):H^{1,1/2}(Q;\partial_t-\Delta_\x)\to H^{-1/2,-1/4}(\Sigma)$ of these potentials as follows
\begin{align*}
	\partial_\n(\widetilde \M_0 U_0) =: \M_1 U_0, \quad
	\partial_\n(\widetilde \V \phi) - \phi/2 =: {\mathscr N} \phi, \quad
	-\partial_\n(\widetilde \K u) =: \W u .
\end{align*}
Applying the normal derivative to \eqref{eq:representation} thus results in
\begin{align}\label{eq:direct}
	\W u =  (1/2-{\mathscr N}) \phi - \M_1 U_0 ,
\end{align}
i.e., \eqref{eq:hypersing operator} with $f:= (1/2-{\mathscr N}) \phi - \M_1 U_0$.
 For $U_0\in L^2(\Omega)$ and sufficiently smooth $\phi$, one has the identities
$$(\M_1 U_0)(t,\x)=\int_\Omega \partial_{\n(\x)}G(t,\x-\y) U_0(\y) \d\y \quad \text{for all } (t,\x) \in \Sigma$$ and
 $$({\mathscr N} \phi)(t,\x)= \int_\Sigma \partial_{\n(\x)}G(t-s,\x-\y)\phi(s,\y)\d\y\d s \quad \text{for all } (t,\x) \in \Sigma.$$
As the hyper-singular operator $\W$ is coercive, i.e.,
\begin{align}\label{eq:coercivity}
	\dual{\W v}{v}_\Sigma \ge c_{\rm coe} \norm{v}{H^{1/2,1/4}(\Sigma)}^2 \quad \text{for all }v\in H^{1/2,1/4}(\Sigma)
\end{align}
with some constant $c_{\rm coe}>0$, \eqref{eq:direct} is uniquely solvable and the solution $u$ is just the missing lateral trace $U|_\Sigma$ to compute $U$ via the representation formula~\eqref{eq:representation}.

Alternatively, one can make the ansatz $U=\widetilde\M_0 U_0 - \widetilde\K u$.
Indeed, both $\widetilde\M_0 U_0$ and $\widetilde\K u$ satisfy the heat equation, where $\widetilde\M_0U_0$ restricted to $\{0\}\times\Omega$ coincides with $U_0$ and  $\widetilde\K u$ vanishes there.
To satisfy the Dirichlet boundary conditions, one has to solve
\begin{align}\label{eq:indirect}
	\W u =  \phi - \M_1 U_0,
\end{align}
i.e., \eqref{eq:hypersing operator} with $f:=\phi - \M_1 U_0$.
While \eqref{eq:direct} is called direct method, as it directly provides the physically relevant quantity $u=U|_\Sigma$, \eqref{eq:indirect} is called indirect method.

For more details and proofs, we refer to the seminal works \cite{an87,noon88,costabel90}, which consider $U_0=0$, and to \cite{dns19,dohr19} for the general case.

We finally mention the additional regularities $\W:\H^{1,1/2}(\Sigma)\to L^2(\Sigma)$ and $\mathscr{N}:L^2(\Sigma)\to L^2(\Sigma)$, stated in \cite[Section 4]{costabel90} for piecewise smooth domains $\Omega$; see also \cite[Remark~3.1.18]{ss11} for the additional regularity of the spatial trace operator. The mapping property $\W^{-1}: L^2(\Sigma) \to \widetilde H^{1,1/2}(\Sigma)$ is even satisfied for general Lipschitz domains $\Omega$. 
Moreover, replacing $\Omega$ by $\R^d$ and considering $(t,\x)\in I\times \R^d$ in \eqref{eq:initial potential}, we extend $\widetilde\M_0$ to an operator on $H^1(\R^d)$. 
Parabolic regularity gives that $\widetilde\M_0:H^1(\R^d)\to  H^{2,1}(I\times \R^d)$ with $H^{2,1}(I\times \R^d)$ defined analogously as in Section~\ref{sec:anisotropic spaces}. 
(Indeed, on $\R^d$, this is a simple consequence of the characterization of Sobolev norms in terms of the Fourier transform, e.g., ~\cite[Theorem~3.16]{mclean00}.)
Extending functions $U_0\in H_0^1(\Omega)$ by zero to $H^1(\R^d)$ and restricting $\widetilde\M_0 U_0$ to $Q$, we further see that $\widetilde\M_0:H_0^1(\Omega)\to  H^{2,1}(Q)$ with $H^{2,1}(Q)$ defined as in Section~\ref{sec:anisotropic spaces}. 
Since $\partial_\n(\cdot): H^{2,1}(Q) \to L^2(\Sigma)$, we conclude that $\M_1:H_0^1(\Omega) \to L^2(\Sigma)$.

\begin{remark}
For sufficiently regular $\Gamma$, one can show the regularity $\partial_\n(\cdot): H^{2,1}(Q) \to H^{1/2,1/4}(\Sigma)$. 
This follows for instance using wavelet expansions; see, e.g., \cite{ss09,cr19}.
From interpolation theory, one can then derive that $\M_1:\H^{1/2}(\Omega) \to L^2(\Sigma)$, where $\H^{1/2}(\Omega):=[L^2(\Omega); H_0^1(\Omega)]_{1/2}$.
However, we stress that while $\H^\mu(\Omega):=[L^2(\Omega); H_0^1(\Omega)]_{\mu} = H^\mu(\Omega)$ for $\mu\in[0,1/2)$, this is not the case for $\mu=1/2$; see, e.g., \cite[Theorems~3.33 and 3.40]{mclean00}.
Thus, like $H_0^1(\Omega)$, also $\H^{1/2}(\Omega)$ encodes a notion of boundary conditions. 
\end{remark}

\subsection{Boundary meshes}\label{sec:meshes}
Throughout this work, we consider prismatic meshes $\PP_h$ of $\Sigma$. 
This means that $\PP_h$ is a set of prisms of the form $P = J\times K$ with closed time intervals $J \subset \overline{I}$ and closed simplices $K \subset \Gamma$ that forms a partition of $ \Sigma$ in the sense that $\overline \Sigma = \bigcup_{P\in \PP_h} P$ and $P \cap P'$ has $d$-dimensional measure zero for all $P, P' \in \PP_h$ with $P \neq P'$. In addition, we assume that $\PP_h$ satisfies the following three properties.
\begin{itemize}
\item[{\rm (i)}]  If the intersection $P \cap P'$ of any $P,P'\in\PP_h$ has positive $(d-1)$-dimensional measure, then $ P \cap P'$ is a hyperface of $P$ or $P'$. 
\item[{\rm (ii)}]  If $J'\times K' = P'\in \PP_h$ and $J''\times K'' =  P'' \in \PP_h$ are temporal neighbors of $J\times K=P\in \PP_h$ on the same side, i.e., if $P \cap P'$ and $P \cap P''$ are $(d-1)$-dimensional subsets of the same hyperface $\{t\}\times K$ of $P$, then $J' = J''$. 
\item[{\rm (iii)}] For almost every $t\in I$, the spatial triangulation $\set{K}{J\times K\in \PP_h, t\in J}$ is conforming (which is trivially satisfied if $d=2$).
\end{itemize}

We further suppose the existence of $\emph{uniform}$ constants $C_\text{shape}, C_\text{lqu}^t, C_\text{lqu}^\x\geq 1$ such that:
\begin{itemize}

\item[{\rm (iv)}]  the spatial parts in $\PP_h$ are shape-regular, i.e., 
\begin{equation}\label{Cshape}
	C_\text{shape}^{-1} |K| \le \diam(K)^{d-1} \le C_\text{shape} |K| \quad \text{for all } J \times K \in \PP_h, 
\end{equation}
(which is trivially satisfied if $d=2$);

\item[{\rm (v)}] $\PP_h$ is locally quasi-uniform, i.e.,
\begin{align}\label{Cquasiunif}
	|J| \le C_\text{lqu}^t |J'| \quad \text{and} \quad \diam(K) \le C_\text{lqu}^\x \diam(K')\quad 
\end{align}
for all $J \times K, J' \times K' \in \PP_h$ with $(J \times K) \cap (J' \times K') \neq \emptyset$.
\end{itemize}
In particular, the reliability constant of \eqref{eq:intreliability} will depend on these constants; see Theorem~\ref{thm:reliability}  below.
A mesh refinement strategy that, starting from a tensor mesh, guarantees all these properties is given in Section~\ref{sec:adaptiveAlg}.

\subsection{Boundary element method}\label{sec:bem}

Given a finite-dimensional subspace $X_h \subset H^{1/2,1/4}(\Sigma)$, let $u_h\in X_h$ denote the Galerkin discretization of the solution $u$ of the boundary integral equation~\eqref{eq:hypersing operator}, i.e.,
\begin{align}\label{eq:Galerkin}
	\dual{\W u_h}{v_h}_\Sigma = \dual{f}{v_h}_\Sigma
	\quad \text{for all }v_h\in X_h,
\end{align}
which is equivalent to the Galerkin orthogonality
\begin{align}\label{eq:orthogonality}
	\dual{\W(u-u_h)}{v_h}_\Sigma = 0
	\quad \text{for all }v_h\in X_h.
\end{align}
Note that coercivity~\eqref{eq:coercivity} guarantees unique solvability of the latter equations, and the C\'ea lemma applies, i.e.,
\begin{align}\label{eq:cea}
	\norm{u-u_h}{H^{1/2,1/4}(\Sigma)}
	\le \frac{\|\W\|}{c_{\rm coe}}\,\min_{v_h\in X_h} \norm{u-v_h}{H^{1/2,1/4}(\Sigma)},
\end{align}
where $\|\W\|$ denotes the operator norm of $\W:H^{1/2,1/4}(\Sigma) \to H^{-1/2,-1/4}(\Sigma)$.

Given a prismatic mesh as in Section \ref{sec:meshes}, a natural choice of $X_h$ would be the space of all $\PP_h$-piecewise polynomials of some fixed degree $p_t \in \N_0$ in time and $p_x \in \N$ in space that are continuous in space. However, in order to ensure $\W u_h \in L^2(\Sigma)$ for the envisaged {\sl a posteriori} error estimate \eqref{eq:intreliability}, we require $X_h \subset \widetilde H^{1,1/2}(\Sigma)$; see Section \ref{sec:integral equations}. With the space of all continuous $\PP_h$-piecewise polynomials $S^{p_t,p_x}(\PP_h)$ of fixed degree $p_t\in \N$ in time and $p_x \in \N$ in space, we hence consider 
\begin{equation}
	X_h := \widetilde S^{p_t,p_\x}(\PP_h) := \set{v_h \in S^{p_t,p_\x}(\PP_h)}{v_h(0,\cdot) = 0}.
\end{equation}
Note that for $f \in L^2(\Sigma)$, e.g., as in \eqref{eq:direct} or \eqref{eq:indirect} with $\phi\in L^2(\Sigma)$ and $U_0 \in H^1_0(\Omega)$, we have the additional regularity $u \in \widetilde H^{1,1/2}(\Sigma)$ (see Section \ref{sec:integral equations}) so that it makes sense to approximate $u$ by functions vanishing at $t = 0$.

If $\PP_h = \set{J\times K}{J\in \PP_{h_t}, K\in\PP_{h_\x}}$ is a full tensor mesh corresponding to a mesh $\PP_{h_\x}$ of $\Gamma$ with uniform mesh size $h_\x\eqsim \diam(K)$ for all $K\in\PP_{h_\x}$ and a mesh $\PP_{h_t}$ of $\overline I$ with uniform step size $h_t\eqsim h_\x^{\sigma}$ for some $\sigma>0$, then, \cite[Proposition~5.3]{costabel90}\footnote{To be precise, the given reference only applies for $X_h = S^{p_t,p_\x}(\PP_h)$, but the proof directly extends to $X_h = \widetilde S^{p_t,p_\x}(\PP_h)$.} suggests the error decay rate
\begin{align}\label{eq:rates}
	\min_{v_h\in\widetilde S^{p_t,p_\x}(\PP_h)} \norm{u-v_h}{H^{1/2,1/4}(\Sigma)} \lesssim N_h^{-\frac{\min\{p_\x+1/2,(p_t+3/4)\sigma\}}{d-1+\sigma}}
	\quad\text{for all smooth }u \text{ with }u(0,\cdot)=0;
\end{align}
see also \cite[Theorem 3.3]{cr19} for Dirichlet boundary value problems.
Here, $N_h\eqsim h_\x^{-(d-1)} h_t^{-1} = h_\x^{-d+1-\sigma}$ denotes the number of degrees of freedom in $X_h$.
The optimal grading parameter is thus given by $\sigma = (p_\x+\tfrac12)/(p_t+\tfrac34)$ with resulting rate $\OO\big(N_h^{-\frac{p_\x+1/2}{d-1+\sigma}}\big)$. Similarly, we expect the decay rate 
 $\OO\big(N_h^-{\frac{\min\{p_\x+1, (p_t+1)\sigma\}}{d-1+\sigma}}\big)$
for the minimal $L^2(\Sigma)$-error.

\subsection{Partition of unity}\label{sec:partitionofunity}

We introduce a basis of the spaces $S^{1,1}(\PP_h)$ and $\widetilde S^{1,1}(\PP_h)$.
The sets of free nodes for $S^{1,1}(\PP_h)$ and $\widetilde S^{1,1}(\PP_h)$ are defined as 
\[
\NN_h := \{\z\in \overline{\Sigma} : \z \text{ is a vertex of every } P\in\PP_h \text{ with } \z \in P\} \quad \text{and}\quad \widetilde{\NN_h}:=\NN_h \setminus (\{0\}\times\Gamma),
\]
respectively. Any other vertex that is not in $\NN_h$ is called hanging node, and the corresponding set is denoted by $\NN_h^\perp$.
For each free node $\z\in\NN_h$, the nodal basis function $\varphi_{h,\z} \in S^{1,1}(\mathcal{P}_h)$ is characterized by
\begin{equation}\label{eq:kronecker}
\varphi_{h,\z}(\z^\prime)=\delta_{\z,\z^\prime} \quad\text{ for all } \z^\prime\in \NN_h.
\end{equation}
Proposition~\ref{prop:varphi} states that $\varphi_{h,\z}$ is indeed well defined by \eqref{eq:kronecker} and provides important properties of both $\varphi_{h,\z}$ as well as its corresponding support $\omega_h(\z):=\supp(\varphi_{h,\z})$. 
The proof is given in Appendix~\ref{appendix:Prop}.

\begin{proposition}\label{prop:varphi}
There hold the following properties {\rm (i)--(iii)}:
\begin{enumerate}
\item[ {\rm (i)}] For each $\z\in\NN_h$, there exists a unique function $\varphi_{h,\z} \in S^{1,1}(\mathcal{P}_h)$ satisfying \eqref{eq:kronecker}, where $\varphi_{h,\z}(t,\x)\ge 0$ for all $(t,\x)\in \Sigma$. 
\item[ {\rm (ii)}] For each $\z\in\NN_h$, $\omega_h(\z)$ consists of a uniformly bounded number of elements $P\in \PP_h$. Conversely, each $P\in\PP_h$ is contained in a uniformly bounded number of patches $\omega_h(\z)$ for $\z\in \NN_h$. Both bounds depend only on $C_{\rm shape}$, $C_{\rm lqu}^t$, and $C_{\rm lqu}^\x$ from \eqref{Cshape}--\eqref{Cquasiunif}.
\item[ {\rm (iii)}] The set $\set{\varphi_{h,\z}}{\z\in \NN_h}$ is a basis of $S^{1,1}(\PP_h)$, and the set $\set{\varphi_{h,\z}}{\z\in \widetilde\NN_h}$ is a basis of $\widetilde S^{1,1}(\PP_h)$. In particular, it holds that
\begin{equation}\label{eq:partitionofunity}
\sum_{\z\in\NN_h} \varphi_{h,\z} = 1.
\end{equation}
\end{enumerate}
\end{proposition}

\section{A posteriori error estimation}\label{sec:aposteriori}

The proof of the {\sl a posteriori} error estimate \eqref{eq:intreliability} requires some preparations. 
In Section \ref{sec:poincare}, we prove Poincaré-type inequalities. These are subsequently used to derive an approximation property of the interpolation operator introduced in Section \ref{sec:interpolop}. 
With this interpolation operator, we conclude the proof of \eqref{eq:intreliability} in Section  \ref{sec:aposterioriproof}.

\subsection{Poincar\'e-type inequalities}\label{sec:poincare}

The following lemma is implicitly stated in the proof of \cite[Proposition 5.3]{costabel90}, while the proposition itself states the result for the full $H^{1/2,1/4}(\Sigma)$-norm.  We give a detailed proof for the sake of completeness.
\begin{lemma}\label{lem:parabolicpoincare}
Let $P=J\times \omega$ with a subinterval $J$ of $I$ and a $(d-1)$-dimensional subset $\omega$ of $\Gamma$. Then, for all $v\in H^{1/2,1/4}(\Sigma)$, it holds that 
\begin{equation}\label{eq:parabolicpoincare}
\|v-v_P\|^2_{L^2(P)} \leq 2\frac{\operatorname{diam}(\omega)^{d-1}}{|\omega|} {\rm diam}(\omega)\,|v|_{L^2(J;H^{1/2}(\omega))}^2
+ 2|J|^{1/2}\,|v|_{H^{1/4}(J;L^2(\omega))}^2 ,
\end{equation}
where $v_P:=|P|^{-1}\int_J \int_\omega v(t,\x) \d\x\d t$.
\end{lemma}
\begin{proof} 
For  (almost) all $\x\in \omega$ and $t\in J$, define
\begin{align*}
\Pi_J v(\cdot,\x):= \frac{1}{|J|}\int_J v(t,\x) \d t,\quad
\Pi_\omega v(t,\cdot):= \frac{1}{|\omega|}\int_\omega v(t,\x) \d \x.
\end{align*}
Then, it holds that $v_P = \Pi_J \Pi_\omega v= \Pi_\omega \Pi_J v,$ which yields that
\begin{equation*}
\|v-v_P\|_{L^2(P)}\leq \|v-\Pi_\omega v\|_{L^2(P)}+\|\Pi_\omega v-\Pi_\omega\Pi_J v\|_{L^2(P)}.
\end{equation*}
In the following two steps, we estimate the two summands to conclude  \eqref{eq:parabolicpoincare}.

\textbf{Step 1:} The Cauchy–Schwarz inequality shows that
\begin{align*}
\|v-\Pi_\omega v\|_{L^2(P)}^2
&=\int_J\int_\omega\Biggl|\frac{1}{|\omega|}\int_\omega v(s,\y)-v(s,\x)\d\x\Biggr|^2 \d\y \d s s\\
&\leq \frac{1}{|\omega|} \int_J\int_\omega\int_\omega |v(s,\y)-v(s,\x)|^2\d\x \d\y\d s.
\end{align*}
Shape regularity \eqref{Cshape} and $|\x-\y|\leq \operatorname{diam}(\omega)$ for all $\x,\y\in \omega$ imply that
\begin{align*}
\|v-\Pi_\omega v\|_{L^2(P)}^2
&\leq \frac{\operatorname{diam}(\omega)^{d-1}}{|\omega|}\operatorname{diam}(\omega)^{1-d}
\int_J\int_\omega\int_\omega \frac{|\x-\y|^d}{|\x-\y|^d}\,|v(s,\y)-v(s,\x)|^2\d\x \d\y\d s \\
&\leq \frac{\operatorname{diam}(\omega)^{d-1}}{|\omega|}\operatorname{diam}(\omega)\,|v|_{L^2(J;H^{1/2}(\omega))}^2.
\end{align*}

\textbf{Step 2:} Since $\Pi_\omega$ is an orthogonal projection, a similar calculation as in Step 1 gives that
\[\|\Pi_\omega v-\Pi_\omega\Pi_J v\|^2_{L^2(P)}\leq \|v-\Pi_J v\|_{L^2(P)}^2\leq  |J|^{1/2}\,|v|_{H^{1/4}(J;L^2(\omega))}^2,\]
which concludes the proof.
\end{proof}
In order to prove a suitable approximation property for the interpolation operator constructed in Section \ref{sec:interpolop} below, we require a Poincar\'e-type inequality on more general local subsets of $\Sigma$. We follow the arguments of \cite[Lemma 5.4 and Remark 5.5]{dss25}, where such a generalization is provided for the space $H^{1,1/2}(Q)$ instead of $H^{1/2,1/4}(\Sigma)$. 
For a measurable set $\omega \subset \Sigma$ and $v \in H^{1/2,1/4}(\Sigma)$, we define the localized seminorms
\begin{align*}
|v|_{L^2H^{1/2}(\omega)}^2
&:=\int_I\int_\Gamma\int_\Gamma \chi_{\omega}(s,\x)\,\chi_{\omega}(s,\y)\,
\frac{|v(s,\x)-v(s,\y)|^2}{|\x-\y|^d} \d\y \d\x\d s,\\[4pt]
|v|_{H^{1/4}L^2(\omega)}^2
&:=\int_I\int_I\int_\Gamma \chi_{\omega}(t,\x)\,\chi_{\omega}(s,\x)\,
\frac{|v(t,\x)-v(s,\x)|^2}{|t-s|^{3/2}}\d\x \d s\d t,
\end{align*}
where $\chi_{\omega}\colon\Sigma\to\{0,1\}$ denotes the characteristic function on $\omega$.
For all $P\in \PP$, we consider the following local subsets
\begin{equation}\label{eq:patchomegaP}
\omega_h(P) := \bigcup_{\z\in \NN_h: P\subseteq \supp(\varphi_{h,\z})} \omega_h(\z).
\end{equation}
From the Poincar\'e-type inequality on $\omega_h(P)$, we then also derive a local estimate  for $\| v \|_{L^2(\omega_h(P))}$ if  $P$  is close to $\{0\}\times \Gamma$.

\begin{lemma}\label{lem:genparabolicpoincare}
There exists a constant $C_{\rm poinc}>0$ depending only on $C_{\rm shape}$, $C_{\rm lqu}^t$, and $C_{\rm lqu}^\x$ from \eqref{Cshape}--\eqref{Cquasiunif} such that for all $v\in H^{1/2,1/4}(\Sigma)$ and all $P=J\times K\in\PP_h$, it holds that
\begin{equation}\label{eq:genparabolicpoincare}
\|v-v_{\omega_h(P)}\|^2_{L^2(\omega_h(P))}
\leq C_{\rm poinc}\Bigl(\operatorname{diam}(K)\,|v|^2_{L^2H^{1/2}(\omega_h(P))} + |J|^{1/2}\,|v|^2_{H^{1/4}L^2(\omega_h(P))}\Bigr),
\end{equation}
where $v_{\omega_h(P)}:=|\omega_h(P)|^{-1}\int_{\omega_h(P)} v(t, \x)\d\x\d t$.

There further exists a constant $C_0>0$ depending only on \eqref{Cshape}--\eqref{Cquasiunif} such that for all $v \in H^{1/2,1/4}(\Sigma)$ and all $P = J \times K \in \PP_h$ with $P\subseteq \omega_h(\z)$ for some $\z\in\NN_h\cap(\{0\}\times \Gamma)$, it holds that
\begin{equation}\label{eq:genparabolicpoincareboundary}
\|v\|^2_{L^2(\omega_h(P))}
\leq C_0\Bigl(\operatorname{diam}(K)\,|v|^2_{L^2H^{1/2}(\omega_h(P))} 
+ |J|^{1/2}\bigl(|v|^2_{H^{1/4}L^2(\omega_h(P))}+\|t^{-1/4}v\|^2_{L^2(\omega_h(P))}\bigr)\Bigr).
\end{equation}
Here, $t^{-1/4}$ is an abbreviation for the mapping $(t,\x) \mapsto t^{-1/4}$.
\end{lemma}
\begin{proof}
We prove the assertion in three steps. In Steps 1--2, we show \eqref{eq:genparabolicpoincare} by combining Lemma~\ref{lem:parabolicpoincare} and the argumentation of \cite[Lemma 5.4]{dss25}. In Step 3, we derive \eqref{eq:genparabolicpoincareboundary}.

\textbf{Step 1:} Let $P\in\PP_h$ be arbitrary. By Proposition~\ref{prop:varphi}~(ii), the patch $\omega_h(P)$ contains a uniformly bounded number of shape-regular elements in $\PP_h$ of comparable size. Thus, there exists a uniformly bounded integer $L$ and an overlapping cover of prisms $\{P_1,\dots,P_L\}$ with
\begin{equation}\label{eq:defgenparabolicpoincare}
\omega_h(P)=\bigcup_{\ell=1}^L P_\ell
\qquad\text{and}\qquad
|P_j|\eqsim\Bigl|\bigcup_{\ell=1}^{j-1}(P_\ell\cap P_j)\Bigr|\quad \text{for all }j=2,\dots,L
\end{equation}
as well as
\begin{equation}\label{helper1genpp}
|J|\eqsim |J_\ell|\quad\text{and}\quad |K|\eqsim |\omega_\ell|\eqsim \diam(\omega_\ell)^{d-1} \quad\text{for all } P_\ell = J_\ell\times \omega_\ell \text{ and }\ell = 1,\ldots,L. 
\end{equation} 
More precisely, the existence of such a cover can be seen as follows. Let $ \{ \widetilde{P}_1,\ldots, \widetilde{P}_M \} $ be the coarsest conforming refinement of $\PP_h$ restricted to $\omega_h(P)$ into prisms (with simplicial base area). 
Our assumptions on $\PP_h$ guarantee that
\begin{equation*}
|J|\eqsim |\widehat{J}_\ell|\quad\text{and}\quad |K|\eqsim |\widehat{K}_\ell|\eqsim \diam(\widehat{K}_\ell)^{d-1} \quad\text{for all } \widehat{P}_\ell = \widehat{J}_\ell\times \widehat{K}_\ell \text{ and }\ell = 1,\ldots,M. 
\end{equation*} 
Let $\widehat{P}_{\pi_1},\widehat{P}_{\pi_2}, \widehat{P}_{\pi_3},\ldots,\widehat{P}_{\pi_L}$ be a path of neighboring prisms (in the sense that $\widehat P_{\pi_{\ell-1}} \cap \widehat P_{\pi_\ell}$ has $(d-1)$-dimensional measure larger than 0 for $\ell=2,\ldots,L$) such that each prism $\widehat{P}_1,..., \widehat{P}_M$ appears at least once. We define $P_1 := \widehat{P}_{\pi_1}$ and $P_\ell:=\widehat{P}_{\pi_{\ell-1}}\cup \widehat{P}_{\pi_\ell}$ for $\ell=2,\ldots,L$. This overlapping cover of prisms fulfills \eqref{eq:defgenparabolicpoincare}--\eqref{helper1genpp}. 

\textbf{Step 2:} 
Define $P^- := \bigcup_{\ell=1}^{L-1}P_\ell$. The best-approximation property of $v_{\omega_h(P)}$ and the definition of $P^-$ give that
\begin{equation}\label{eq:lemhelppoincare}
\|v-v_{\omega_h(P)}\|_{L^2(\omega_h(P))}^2
\leq \|v-v_{P^-}\|_{L^2(\omega_h(P))}^2
\leq \|v-v_{P^-}\|_{L^2(P^-)}^2 + \|v-v_{P^-}\|_{L^2(P_L)}^2.
\end{equation}
Since $v - v_{P_L}$ is $L^2(P_L)$-orthogonal to constants, the second summand can be rewritten as 
\begin{equation}\label{eq:lemhelp2poincare}
\|v-v_{P^-}\|_{L^2(P_L)}^2
= \|v-v_{P_L}\|_{L^2(P_L)}^2 + \|v_{P_L}-v_{P^-}\|_{L^2(P_L)}^2.
\end{equation}

Let $P_L^{\mathrm{int}}:=P^-\cap P_L$. With the Cauchy--Schwarz inequality and the equivalence of $|P_L|$ and $|P_L^{\mathrm{int}}|$ from \eqref{eq:defgenparabolicpoincare}, we get that
\begin{align*} \|v_{P^-}-v_{P_L}\|_{L^2(P_L)}^2&\leq 2\,\|v_{P_L^\text{int}}-v_{P_L}\|_{L^2(P_L)}^2+2\,\|v_{P_L^\text{int}}-v_{P^-}\|_{L^2(P_L)}^2\\
 &=2 \,|P_L|\biggl(\Bigl|\frac{1}{|P_L^\text{int}|} \int_{P_L^\text{int}}v(t, \x)-v_{P_L} \d\x \d t\Bigl|^2+\Bigl|\frac{1}{|P_L^\text{int}|} \int_{P_L^\text{int}}v(t, \x)-v_{P^-} \d\x\d t\Bigl|^2\biggr)\\
  &\leq 2 \,\frac{|P_L|}{|P_L^\text{int}|}\biggl( \int_{P_L^\text{int}}|v(t, \x)-v_{P_L}|^2 \d\x \d t+\int_{P_L^\text{int}}|v(t, \x)-v_{P^-}|^2 \d\x\d t \biggr)\\
  &\lesssim \|v-v_{P_L}\|^2_{L^2(P_L^\text{int})}+\|v-v_{P^-}\|^2_{L^2(P_L^\text{int})}\\
	  &\leq \|v-v_{P_L}\|^2_{L^2(P_L)}+\|v-v_{P^-}\|^2_{L^2(P^-)}. \end{align*} 
Combining this with \eqref{eq:lemhelppoincare}--\eqref{eq:lemhelp2poincare} yields that
\[
\|v-v_{\omega_h(P)}\|_{L^2(\omega_h(P))}^2
\lesssim \|v-v_{P_L}\|^2_{L^2(P_L)} + \|v-v_{P^-}\|^2_{L^2(P^-)}.
\]
Since $L$ is uniformly bounded, we inductively see that
\[
\|v-v_{\omega_h(P)}\|_{L^2(\omega_h(P))}^2
\lesssim \sum_{\ell=1}^L \|v-v_{P_\ell}\|_{L^2(P_\ell)}^2.
\]
Applying Lemma~\ref{lem:parabolicpoincare} to each $P_\ell = J_\ell\times \omega_\ell$ and using the equivalence \eqref{helper1genpp} yield \eqref{eq:genparabolicpoincare}.

\textbf{Step 3:} It remains to show \eqref{eq:genparabolicpoincareboundary}. Assume that $P\subseteq\omega_h(\z)$ for some $\z\in \NN_h\cap(\{0\}\times\Gamma)$.
From $L^2(\omega_h(P))$-orthogonality of $v-v_{\omega_h(P)}$ to constants, we first see that
\[
\|v\|_{L^2(\omega_h(P))}^2
= \|v-v_{\omega_h(P)}\|_{L^2(\omega_h(P))}^2 + \|v_{\omega_h(P)}\|_{L^2(\omega_h(P))}^2.
\]
The first summand is controlled by \eqref{eq:genparabolicpoincare}.
For the second summand, we use that the number of elements contained in $\omega_h(P)$ is uniformly bounded with
\begin{equation}\label{eq:lemhelp3poincare}
|J| \eqsim |J'|\quad\text{and}\quad |K|\eqsim |K'| \quad\text{for all} \quad P'=J'\times K'\in \PP_h \text{ with } P'\subseteq\omega_h(P),
\end{equation}
which is a simple consequence of Proposition~\ref{prop:varphi}~(ii), to see that
\begin{align*}
\|v_{\omega_h(P)}\|_{L^2(\omega_h(P))}^2 &= \frac{|\omega_h(P)|}{|P|}\|v_{\omega_h(P)}\|_{L^2(P)}^2
\lesssim \|v_{\omega_h(P)}\|_{L^2(P)}^2
\leq 2\|v-v_{\omega_h(P)}\|^2_{L^2(P)}+2\|v\|^2_{L^2(P)}.
\end{align*}
The first summand is again controlled by \eqref{eq:genparabolicpoincare}. For the remaining term $\|v\|^2_{L^2(P)}$, we introduce the weight $t^{-1/2}$. The equivalence \eqref{eq:lemhelp3poincare} and the fact that at least one element contained in $\omega_h(P)$ has nonempty intersection with $\{0\}\times \Gamma$ imply that $t^{1/2}\lesssim |J|^{1/2}$ for all $t\in J$, which leads to
\begin{align*}
\|v\|^2_{L^2(P)}
&= \int_{J} t^{1/2}\,t^{-1/2}\,\|v(t,\cdot)\|_{L^2(K)}^2\d t \\
& \lesssim  |J|^{1/2}\int_{J} t^{-1/2}\|v(t,\cdot)\|_{L^2(K)}^2\d t
\leq  |J|^{1/2}\|t^{-1/4}v\|_{L^2(\omega_h(P))}^2.
\end{align*}
 This concludes the proof.
\end{proof}

\subsection{Interpolation operator}\label{sec:interpolop}

The crucial ingredient in the proof of the {\sl a posteriori} error estimate \eqref{eq:intreliability} is a suitable interpolation operator $\mathcal{I}_h: H^{1/2,1/4}(\Sigma) \to \widetilde S^{1,1}(\PP_h)$. We introduce a Clément--type interpolation operator $\mathcal{I}_h\colon L^2(\Sigma)\to \widetilde S^{1,1}(\PP_h)$ by
\begin{equation}\label{Interpolationop}
\mathcal{I}_h v:= 
\sum_{\z\in\widetilde{\NN}_h}
\left( \frac{1}{|\omega_h(\z)|}\int_{\omega_h(\z)}v(t,\x)\d\x\d t \right)\varphi_{h,\z}.
\end{equation}

For the crucial local approximation property of Lemma~\ref{lem:interp-stability} below, we first need to show the following elementary properties of $\mathcal{I}_h$.
\begin{lemma}\label{lem:Interpolationop}
On all elements $P\in\PP_h$ that are away from $\{0\}\times \Gamma$ in the sense that $P \not\subseteq \omega_h(\z)$ for all $\z\in \NN_h\setminus\widetilde{\NN_h}$, the interpolation operator $\mathcal{I}_h$ preserves constants, i.e.,
\begin{equation}\label{eq:Ihproj}
(\mathcal{I}_h c)|_P = c \quad \text{for all } c\in\R.
\end{equation}
Moreover, $\mathcal{I}_h$ is locally $L^2$-stable, i.e., there exists a constant $C_{\rm stab}>0$ depending only on $C_{\rm shape}$, $C_{\rm lqu}^t$ and $C_{\rm lqu}^\x$ from \eqref{Cshape}--\eqref{Cquasiunif} such that
\begin{equation}\label{eq:Ihstab}
\|\mathcal{I}_h v\|_{L^2(P)} \leq C_{\rm stab}\,\|v\|_{L^2(\omega_h(P))}
\quad \text{for all } v\in L^2(\Sigma) \text{ and } P\in\PP_h.
\end{equation}
\end{lemma}

\begin{proof}
For all $c\in\R$, the defintion of $\mathcal{I}_h$ and the partition of unity \eqref{eq:partitionofunity} give that
\[\mathcal{I}_h c = c \sum_{\z\in\widetilde{\NN}_h}\varphi_{h,\z} = c\, \Bigl(1-\sum_{\z\in\NN_h\setminus\widetilde{\NN_h}}\varphi_{h,\z}\Bigl),\]
which, together with the assumption on $P$, leads immediately to \eqref{eq:Ihproj}.

To show that $\mathcal{I}_h$ is locally $L^2$-stable, we apply the triangle inequality, the fact that  all $\varphi_{h,\z}$ take values in $[0,1]$ by Proposition~\ref{prop:varphi}~(i)+(iii), and the Cauchy--Schwarz inequality,
\begin{align*}
\|\mathcal{I}_h v\|_{L^2(P)}
&\leq \sum_{\z\in\widetilde{\NN}_h: P\subseteq \omega_h(\z)}
\frac{1}{|\omega_h(\z)|}\left|\int_{\omega_h(\z)}v(t,\x)\d t\d\x\right|\|\varphi_{h,\z}\|_{L^2(P)} \\
&\leq \sum_{\z\in\widetilde{\NN}_h: P\subseteq  \omega_h(\z)}
|\omega_h(\z)|^{-1/2} \|v\|_{L^2(\omega_h(\z))}\,|P|^{1/2}\\
&\leq \sum_{\z\in\widetilde{\NN}_h: P\subseteq  \omega_h(\z)}
\|v\|_{L^2(\omega_h(\z))}\lesssim \|v\|_{L^2(\omega_h(P))},
\end{align*}
where the last inequality follows from the fact that every $P\in\PP_h$ is included in a uniformly bounded number of patches $\omega_h(\z)$; see Proposition~\ref{prop:varphi}~(ii).
\end{proof}

For the following approximation property of $\mathcal{I}_h$, we additionally require local parabolic scaling of $\PP_h$ in the sense that there exists a \emph{uniform} constant $C_\text{par} \ge 1$ such that 
\begin{equation}\label{eq:localparabolicscaling}
	C_\text{par}^{-1} \diam(K)^2 \le |J| \le C_\text{par} \diam(K)^2 \quad \text{for all } J \times K \in \PP_h. 
\end{equation}
To ease notation, we introduce the local spatial mesh-size function $h_\x \in L^\infty(\Sigma)$, which is defined $\PP_h$-piecewise via 
\begin{equation}
	h_\x|_{J \times K} := \diam(K) \quad \text{for all } J \times K \in \PP_h.
\end{equation}

\begin{proposition}\label{lem:interp-stability}
There exists a constant $C_{\rm app}>0$ depending only on $C_{\rm shape}$, $C_{\rm lqu}^t$, and $C_{\rm lqu}^\x$ from \eqref{Cshape}--\eqref{Cquasiunif}, $C_{\rm par}$ from \eqref{eq:localparabolicscaling}, and on the final time $T$ such that
\begin{equation}\label{eq:interpolationieq}
\bigl\|h_\x^{-1/2}(v-\mathcal{I}_h v)\bigr\|_{L^2(\Sigma)}
\le C_{\rm app}\,\|v\|_{H^{1/2,1/4}(\Sigma)}
\qquad\text{for all } v\in H^{1/2,1/4}(\Sigma).
\end{equation}
\end{proposition}

\begin{proof}
We prove the assertion in two steps. The first step is devoted to establishing a local approximation property for any 
$P = J\times K\in\PP_h$, and in the second step, we derive the global approximation property \eqref{eq:interpolationieq}.

\textbf{Step 1:} First, we assume that that $P$ is away from $\{0\}\times \Gamma$ in the sense that $P\not\subseteq \omega_h(\z)$ for all $\z\in\NN_h\cap(\{0\}\times\Gamma)$. Then, Lemma~\ref{lem:Interpolationop} guarantees that $(\mathcal{I}_h c)|_P=c$ for every constant $c\in\mathbb R$.
The triangle inequality, the local $L^2$-stability\eqref{eq:Ihstab}, and the Poincaré-type inequality \eqref{eq:genparabolicpoincare} show that
\begin{align*}
\|v-\mathcal{I}_h v\|^2_{L^2(P)}
&\lesssim \|v-v_{\omega_h(P)}\|^2_{L^2(P)} + \|\mathcal{I}_h(v_{\omega_h(P)}-v)\|^2_{L^2(P)}\\
&\lesssim \|v-v_{\omega_h(P)}\|^2_{L^2(\omega_h(P))}\\
&\lesssim \operatorname{diam}(K)\,|v|_{L^2H^{1/2}(\omega_h(P))}^2
+ |J|^{1/2}\,|v|_{H^{1/4}L^2(\omega_h(P))}^2.
\end{align*}

 We now assume that $P\subseteq \omega_h(\z)$ for some $\z\in\NN_h\cap(\{0\}\times\Gamma)$. The triangle
inequality, the local $L^2$-stability \eqref{eq:Ihstab}, and the local estimate
\eqref{eq:genparabolicpoincareboundary} show that
\[
\|v-\mathcal{I}_h v\|_{L^2(P)}^2\lesssim \| v \|_{L^2(\omega_h(P))}^2
\lesssim \operatorname{diam}(K)\,|v|_{L^2H^{1/2}(\omega_h(P))}^2
+ |J|^{1/2}\bigl(|v|_{H^{1/4}L^2(\omega_h(P))}^2+\|t^{-1/4}v\|_{L^2(\omega_h(P))}^2\bigr).
\]
 
With local parabolic scaling \eqref{eq:localparabolicscaling}, we obtain in either case the local bound
\begin{equation}\label{eq:interpolationhelper1}
\bigl\|h_\x^{-1/2}(v-\mathcal{I}_h v)\bigr\|_{L^2(P)}^2
\lesssim |v|_{L^2H^{1/2}(\omega_h(P))}^2 + |v|_{H^{1/4}L^2(\omega_h(P))}^2
+ \|t^{-1/4}v\|_{L^2(\omega_h(P))}^2.
\end{equation}

\textbf{Step 2:} To handle the weighted term in \eqref{eq:interpolationhelper1}, we use the following inequality from
\cite[Lemma 3.31]{mclean00}
\[
\int_I t^{-1/2} w(t)^2\d t \lesssim \|w\|_{H^{1/4}(I)}^2~~~~\text{for all } w\in H^{1/4}(I),
\]
where the hidden constant depends only on $T$. This gives that
\begin{equation}\label{eq:mclean-mod}
\|t^{-1/4}v\|_{L^2(\Sigma)}^2 \lesssim \|v\|_{H^{1/4}(I;L^2(\Gamma))}^2.
\end{equation}

Together with \eqref{eq:interpolationhelper1}, we conclude that
\begin{align*}\| h_\x^{-1/2} (v - \mathcal{I}_h v) \|_{L^2(\Sigma)}^2 
&= \sum_{P \in \PP_h} \| h_\x^{-1/2} (v - \mathcal{I}_h v) \|_{L^2(P)}^2\\
&\lesssim \sum_{P\in \PP_h} |v|_{L^2H^{1/2}(\omega_h(P))}^2 + |v|_{H^{1/4}L^2(\omega_h(P))}^2
+ \|t^{-1/4}v\|_{L^2(\omega_h(P))}^2\\
&\lesssim \| v \|_{H^{1/2,1/4}(\Sigma)}^2,
\end{align*} where the last inequality follows from the definition of the localized seminorms $|\cdot|_{L^2H^{1/2}(\omega_h(P))}$ and $|\cdot|_{H^{1/4}L^2(\omega_h(P))}$ and the uniformly bounded overlap of the patches $\omega_h(P)$, which is a consequence of Proposition~\ref{prop:varphi}~(ii).
\end{proof}

\subsection{A posteriori error estimate}\label{sec:aposterioriproof} 

With the preparations of the previous sections, we are finally in the position to prove our main result.
\begin{theorem}\label{thm:reliability}
For given $f \in L^2(\Sigma)$, e.g., as in \eqref{eq:direct} or \eqref{eq:indirect} with $\phi \in L^2(\Sigma)$ and $U_0 \in H^1_0(\Omega)$, let $u \in H^{1/2,1/4}(\Sigma)$ be the solution of \eqref{eq:interior} (which implies, as already mentioned in Section \ref{sec:bem}, the additional regularity $u \in \widetilde H^{1,1/2}(\Sigma)$). Then, for any prismatic mesh $\PP_h$ as in Section \ref{sec:meshes} with local parabolic scaling \eqref{eq:localparabolicscaling} and corresponding Galerkin approximation $u_h \in \widetilde S^{p_t,p_x}(\PP_h)$ of \eqref{eq:Galerkin}, it holds that
\begin{equation}\label{eq:relie}
\|u-u_h\|_{H^{1/2,1/4}(\Sigma)} \le C_{\mathrm{rel}} \,
\bigl\|h_{\x}^{1/2}(f-\W u_h)\bigr\|_{L^2(\Sigma)},
\end{equation}
where $C_{\mathrm{rel}}  = C_{\mathrm{app}} \, \|\W^{-1}\|$ with the constant $C_\mathrm{app}$ from Proposition \ref{lem:interp-stability} and  the operator norm
\[
\|\W^{-1}\|:=\sup_{v\in H^{-1/2,-1/4}(\Sigma)\setminus\{0\}}
\frac{\|\W^{-1}v\|_{H^{1/2,1/4}(\Sigma)}}{\|v\|_{H^{-1/2,-1/4}(\Sigma)}}.
\]
\end{theorem}
\begin{proof}
Using $\W u=f$ and the definition of $\|\W^{-1}\|$, we see that
\[
\|u-u_h\|_{H^{1/2,1/4}(\Sigma)}
\le \|\W^{-1}\|\,\|\W(u-u_h)\|_{H^{-1/2,-1/4}(\Sigma)}
= \|\W^{-1}\|\,\|f-\W u_h\|_{H^{-1/2,-1/4}(\Sigma)}.
\]
By duality, it holds that
\[
\|f-\W u_h\|_{H^{-1/2,-1/4}(\Sigma)}
= \sup_{v\in H^{1/2,1/4}(\Sigma)\setminus\{0\}}
\frac{\langle f-\W u_h, v\rangle_\Sigma}{\|v\|_{H^{1/2,1/4}(\Sigma)}}.
\]
Galerkin orthogonality \eqref{eq:orthogonality} implies that the duality
pairing vanishes on $\widetilde S^{1,1}(\PP_h)$ for $v \in \widetilde S^{1,1}(\mathcal{P}_h) \subseteq \widetilde S^{p_t,p_x}(\mathcal{P}_h)$. Hence, for arbitrary
$v\in H^{1/2,1/4}(\Sigma)$, we may replace $v$ by $v-\mathcal{I}_h v$
in the numerator. With the Cauchy--Schwarz inequality and Lemma~\ref{lem:interp-stability}, this leads to
\begin{align*}
\frac{\langle f-\W u_h, v\rangle_\Sigma}{\|v\|_{H^{1/2,1/4}(\Sigma)}}
&= \frac{\langle f-\W u_h , v-\mathcal{I}_h v\rangle_\Sigma}
{\|v\|_{H^{1/2,1/4}(\Sigma)}}\\
&= \frac{\langle h_\x^{1/2}(f-\W u_h) , h_\x^{-1/2}(v-\mathcal{I}_h v)\rangle_\Sigma}
{\|v\|_{H^{1/2,1/4}(\Sigma)}}\\
&\le \|h_\x^{1/2}(f-\W u_h)\|_{L^2(\Sigma)}
\frac{\|h_\x^{-1/2}(v-\mathcal{I}_h v)\|_{L^2(\Sigma)}}
{\|v\|_{H^{1/2,1/4}(\Sigma)}}\\
&\le C_{\mathrm{app}}\,\|h_\x^{1/2}(f-\W u_h)\|_{L^2(\Sigma)}.
\end{align*}
Taking the
supremum over $v$ yields \eqref{eq:relie}.
\end{proof}

\section{Numerical experiments}

In this section, we employ the error estimator
\begin{equation}\label{eq:estimator}
\eta_h^2:=\sum_{J\times K\in \PP_h} \eta_h(J\times K)^2 \quad \text{with} \quad \eta_h(J\times K):=\text{diam}(K)^{1/2}\|f-\W u_h\|_{L^2(J\times K)}
\end{equation}
within an adaptive algorithm and numerically investigate the resulting convergence rates.
 We restrict ourselves to the case $d=2$, with $\Gamma=\partial \Omega$ being the boundary of a polygonal domain in $\R^2$. 
 Details on the numerical computation of all involved (singular) integrals are given in Appendix~\ref{sec:numericalcomp} for $d\in\{2,3\}$. 
 
\subsection{Adaptive algorithm}\label{sec:adaptiveAlg}

The following adaptive algorithm is applied.
\begin{algorithm}\label{alg}
\textbf{Input:} Initial tensor mesh $\PP_0=\set{J\times K}{J\in \PP_{\overline{I}}, K \in \PP_\Gamma}$ corresponding to a mesh $\PP_{\overline{I}}$ of $\overline{I} = [0,T]$ and a mesh $\PP_\Gamma$ of $\Gamma$, polynomial degrees $p_t$ and $p_\x$, and adaptivity parameter $0<\theta\le1.$\\
\textbf{Loop:} For $\ell = 0,1,2,...$, iterate the following four steps {\rm (i)}--{\rm (iv)}:
\begin{itemize}
\item[{\rm (i)}] \textbf{Solve:} Compute the Galerkin approximation $u_\ell\in \widetilde{S}^{p_t,p_\x}(\mathcal{P}_\ell)$.
\item[{\rm (ii)}] \textbf{Estimate:} For every element $J\times K\in\mathcal{P}_\ell$, compute the local error indicator $\eta_\ell(J\times K)$.
\item[{\rm (iii)}]  \textbf{Mark:} Find a set $\mathcal{M}_\ell \subseteq \mathcal{P}_\ell$ of minimal cardinality such that 
\[\theta^2\eta_\ell\leq \sum_{J\times K\in\mathcal{M}_\ell}\eta_\ell(J\times K)^2.\]
\item[{\rm (iv)}]  \textbf{Refine:} Refine at least all elements in $\mathcal{M}_\ell$ by bisecting them once in space and twice in time (see Figure~\ref{fig:refine}) to obtain a new mesh $\mathcal{P}_{\ell+1}$.
\end{itemize}
\textbf{Output:} Refined prismatic meshes $\PP_\ell$, corresponding discrete solutions $u_\ell \in \widetilde S^{p_t,p_\x}(\PP_\ell)$, and error estimators $\eta_\ell$ for all $\ell \in \N_0$.
\end{algorithm}
\begin{figure}[H]
\centering

    \begin{tikzpicture}
        \draw[thick] (0,0) rectangle (3,3);

        \begin{scope}[shift={(6,0)}]
            \draw[thick] (0,0) rectangle (3,3);
            \foreach \x in {0.75,1.5, 2.25} {
                \draw (\x, 0) -- (\x, 3);
            }
            \draw (0,1.5) -- (3,1.5);
        \end{scope}

        \draw[->, thick] (4.2,1.5) -- (4.8,1.5);  
        \node at (1.5,-0.5) {$t$};
        \node at (7.5,-0.5) {$t$};
        \node at (-0.5,1.5) {$\x$};
        \node at (5.5,1.5) {$\x$};
    \end{tikzpicture}

    \caption{Parabolic refinement of a prismatic element $P=J\times K$, ensuring parabolic scaling \eqref{eq:localparabolicscaling}.}
    \label{fig:refine}
  
\end{figure}
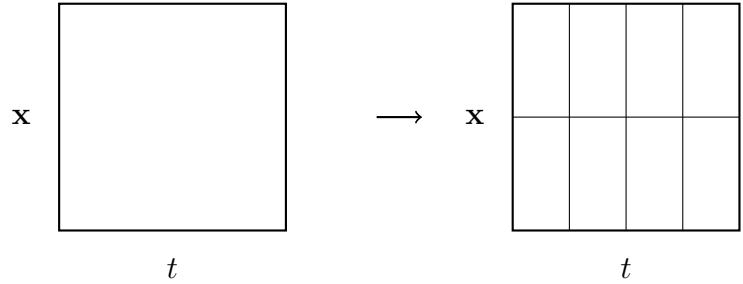
To ensure reliability of the estimator, it is crucial that the refinement step (iv) ensures parabolic scaling \eqref{eq:localparabolicscaling} and uniformly preserves the mesh assumptions of Section \ref{sec:meshes}, which also requires refining non-marked elements. 
We start with an arbitrary tensor mesh $\PP_0$ and assign level 0 to every element in it. With each parabolic refinement (Figure~\ref{fig:refine}), the level of the new elements is increased by one. We choose $\PP_{\ell+1}$ in the refinement step (iv) as the minimal refinement of $\PP_\ell$ such that all elements in $\mathcal{M}_\ell$ are refined and such that the level difference between two elements sharing an edge is at most one.
It follows immediately that all conditions from Section~\ref{sec:meshes} are fulfilled. The constants $C_\text{lqu}^t$ and $C_\text{lqu}^\x$ from~\eqref{Cquasiunif} can be chosen as $C_\text{lqu}^t=16$ and $C_\text{lqu}^\x=4$.

\begin{remark}
In the case $d=3$, the bisection in space in the refinement step {\rm (iv)} can be replaced by newest vertex bisection. Similarly as in the case $d=2$, a minimal number of additional refinements, ensuring that the spatial triangulations at almost every $t \in I$ are conforming and that the level difference of two prisms sharing a face of the type $\{t\}\times K$ is bounded by one, guarantee the properties of Section~\ref{sec:meshes}. More details on such a refinement strategy (for meshes of the space-time cylinder $Q$ instead of the space-time boundary $\Sigma$) are found in \cite[Section 3]{mss26} and \cite[Section 4]{dss25}.
\end{remark}

 \subsection{Reference for exact error}
 
Since the exact error $\|u-u_h\|_{H^{1/2,1/4}(\Sigma)}$ is not available in practice, either because its computation is too expensive or because the exact solution is unknown, we use the following $(h-h/2)$-estimator $\zeta_h$ as a reference quantity instead.  For a prismatic mesh $\PP_h$, let $\widehat{\PP}_h$ denote the uniform isotropic refinement of $\PP_h$, i.e., $\widehat \PP_h$ results from $\PP_h$ by bisecting all elements once in space and once in time. Let $\widehat{u}_h$ be the corresponding Galerkin approximation of the exact solution $u$. The $(h-h/2)$-estimator is defined by
\begin{align}\label{eq:hhalbe}
\norm{u_h-\widehat u_h}{\W} := \Bigl(\langle \W(u_h-\widehat u_h),\, u_h-\widehat u_h\rangle_\Sigma\Bigr)^{1/2}.
\end{align}
Under the saturation assumption
\begin{equation}\label{eq:qsat}
\norm{u-\widehat u_h}{H^{1/2,1/4}(\Sigma)}
\le q_{\rm sat}\, \norm{u-u_h}{H^{1/2,1/4}(\Sigma)},
\qquad 0<q_{\rm sat}<1,
\end{equation}
the triangle inequality and boundedness and coercivity of $\W$ show that this estimator is equivalent to the error $\norm{u-u_h}{H^{1/2,1/4}(\Sigma)}$. Note that the saturation assumption is indeed satisfied under the realistic
(asymptotic) assumption that $\norm{u-u_h}{H^{1/2,1/4}(\Sigma)} = \mathcal O((\# \PP_h)^{-s})$ for some arbitrary rate $s > 0$.

If the exact solution $u$ is known, we additionally compute the $L^2$-error $\norm{u-u_h}{L^2(\Sigma)}$.

\subsection{Experiments}

In the following numerical experiments, we consider the heat equation \eqref{eq:interior} with end time $T=1$ on different two-dimensional domains $\Omega$ and given initial condition $U_0 \in H_0^1(\Omega)$ and Neumann datum $\phi \in L^2(\Sigma)$. We approximate the corresponding Dirichlet datum $u$, being the exact solution of the boundary integral equation \eqref{eq:direct}, by the adaptive Algorithm \ref{alg} with $\theta = 0.5$. 
Note that the estimator is indeed well-defined by the regularity of the data $U_0$ and $\phi$ and the approximations $u_\ell$; see Section \ref{sec:meshes}.
For comparison, we also consider $\theta = 1$, which results in uniform refinement. 
\begin{table}[H]
\centering
\begin{tabular}{c|c|c|c}
 & $p_\x = 1$ & $p_\x = 2$ & $p_\x\geq 3$ \\
\hline
$ \norm{u-u_\ell}{H^{1/2,1/4}(\Sigma)}$
& $\mathcal O(N_\ell^{-1/2})$ 
& $\mathcal O(N_\ell^{-5/6})$ 
& $\mathcal O(N_\ell^{-7/6})$ \\

$\norm{u-u_\ell}{L^2(\Sigma)}$ 
& $\mathcal O(N_\ell^{-2/3})$ 
& $\mathcal O(N_\ell^{-1})$ 
& $\mathcal O(N_\ell^{-4/3})$ \\
\end{tabular}
\caption{Expected convergence rates with respect to the number of degrees of freedom $N_\ell$ of the estimators and $L^2$-error for $p_t = 1$ and $p_\x \in \N$ in case of a smooth solution $u$.} 
\label{table:konvergenz}
\end{table}
In each example, we choose the initial mesh $\PP_0$ consisting of rectangular elements of size 0.5 in both space and time. The employed refinement strategy guarantees parabolically scaled meshes, i.e., $\sigma = 2$, which is required to guarantee reliability of the estimator; see Theorem~\ref{thm:reliability}. In case of smooth solutions $u$, Section~\ref{sec:bem} yields that for $p_t=1$, the optimal choice of $p_\x$ is 3; see also Table~\ref{table:konvergenz}.
We fix the temporal polynomial degree as $p_t:=1$ and consider the lowest-order choice $p_\x :=1$ and the optimal choice $p_\x:=3$ in space.  
	\subsubsection{Smooth problem on the square}
	
	We prescribe the exact solution 
\begin{align}\label{problem:cos}
U(t,\x):=1-e^{-4\pi^2 t}(\cos(2\pi x_1)+\cos(2\pi x_2))+e^{-8\pi^2 t}\cos(2\pi x_1)\cos(2\pi x_2)
\end{align}
on the square $\Omega := (0,1)^2$ and choose the initial condition $U_0$ and the Neumann datum $\phi$ accordingly.  It is noteworthy that the Dirichlet trace converges fast to $1$ as $t$ increases. In the numerical results, displayed in Figure \ref{fig:cos}, we observe optimal rates for both adaptive and uniform refinement. However, owing to the mentioned behavior of the solution $U$ (and its Dirichlet trace $u$), adaptive refinement achieves much better results. In each case, we see a close resemblance between the estimator $\eta_\ell =\norm{h_{\x}^{1/2}(f-\W u_\ell)}{L^2(\Sigma)}$ and the ($h-h/2$)-estimator $\norm{u_\ell - \widehat u_\ell}{\W}$, which numerically indicates both reliability and efficiency of $\eta_\ell$. Note that in this plot and in all following plots, $\norm{u_\ell-\widehat u_\ell}{\W}$ is not always computed over the same range as $\eta_\ell$, owing to its high computational cost.
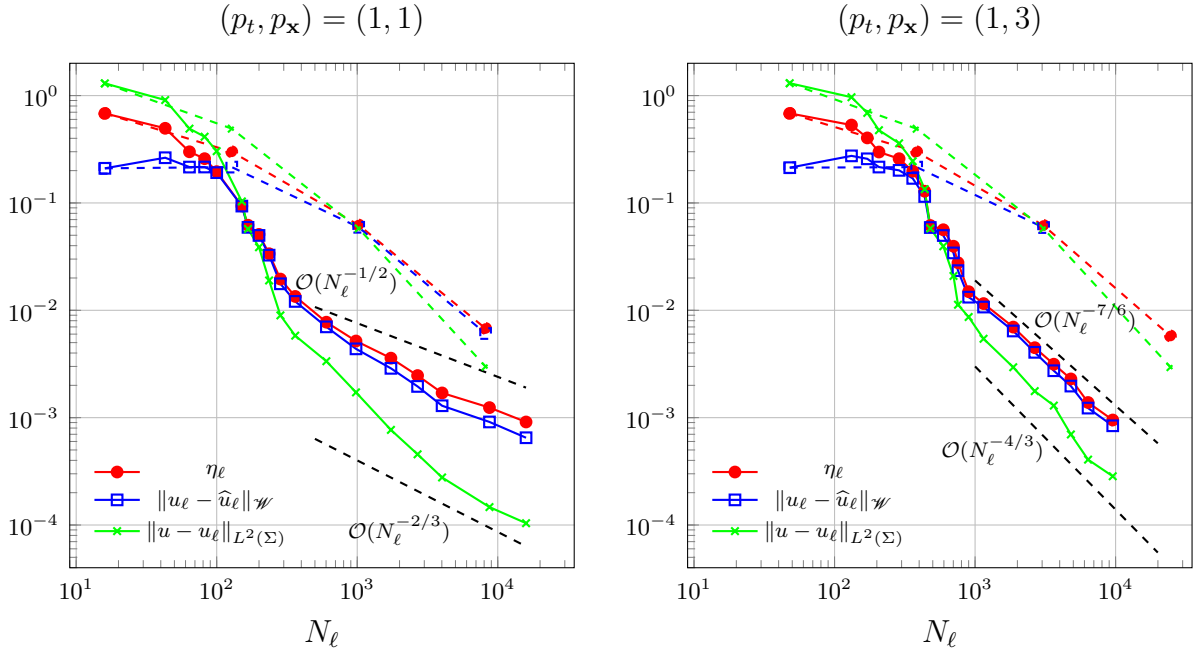
\begin{figure}[H]
\centering
\begin{tikzpicture}
\begin{groupplot}[
  group style={
    group size=2 by 1,
    horizontal sep=1.5cm,
  },
  legend style={
    at={(0.03,0.03)},
    anchor=south west,
    font=\tiny,
    draw=none,
    fill=none
  },
  width=0.5\textwidth,
  height=0.5\textwidth,
  xmode=log,
  ymode=log,
  xmin=9, xmax=35000,
  ymin=4e-5, ymax=2, 
  xmajorgrids, ymajorgrids,
  xlabel={$N_\ell$},
  yticklabel pos=left,
  ytick pos=left,
  tick label style={font=\scriptsize},
]

\nextgroupplot[title={$(p_t, p_\x)=(1,1)$}]
\addplot[red, mark=*, thick]  coordinates {
(16, 0.68209086)
(43, 0.49589549)
(64, 0.29963421)
(82, 0.25847378)
(100, 0.19565523)
(151, 0.09379675)
(167, 0.06202878)
(200, 0.05095717)
(236, 0.03363399)
(284, 0.01954026)
(362, 0.01346055)
(603, 0.00774989)
(980, 0.00517599)
(1737, 0.00358374)
(2688, 0.00246733)
(4017, 0.00169892)
(8730, 0.00124065)
(15899, 0.00091145)
    	};
\addlegendentry{ $\eta_\ell$}

\addplot[dashed, red, mark=*, thick, forget plot]  coordinates {
(16, 0.68209086)
(128, 0.29960823)
(1024, 0.06222659)
(8192, 0.00676949)
    	};
\addplot[blue, mark=square, thick] coordinates {
(16, 0.21018880)
(43, 0.26378123)
(64, 0.21581949)
(82, 0.21615975)
(100, 0.19159422)
(151, 0.09349914)
(167, 0.05925595)
(200, 0.04982080)
(236, 0.03263717)
(284, 0.01765757)
(362, 0.01209657)
(603, 0.00701108)
(980, 0.00437409)
(1737, 0.00287982)
(2688, 0.00195408)
(4017, 0.00129137)
(8730, 0.00091201)
(15899, 0.00065009)
    	};
\addlegendentry{$\norm{u_\ell-\widehat u_\ell}{\W}$}

\addplot[dashed, blue, mark=square, thick, forget plot] coordinates {
(16, 0.21018880)
(128, 0.21577131)
(1024, 0.05936361)
(8192, 0.00607487)
    	};
\addplot[green, mark=x, thick] coordinates {
(16, 1.29916295)
(43, 0.90880922)
(64, 0.49216172)
(82, 0.41401367)
(100, 0.30364195)
(151, 0.10328191)
(167, 0.05727195)
(200, 0.03867474)
(236, 0.01896997)
(284, 0.00896993)
(362, 0.00580486)
(603, 0.00335412)
(980, 0.00171992)
(1737, 0.00076929)
(2688, 0.00045565)
(4017, 0.00027734)
(8730, 0.00014688)
(15899, 0.00010395)
	};
\addlegendentry{$\norm{u-u_\ell}{L^2(\Sigma)}$}

\addplot[dashed, green, mark=x, thick, forget plot] coordinates {
(16, 1.29916295)
(128, 0.49204006)
(1024, 0.05787860)
(8192, 0.00299124)
	};
\addplot[dashed, domain=500:15899, samples=200, thick]
{0.24*x^(-1/2)}
node[pos=0.15, above, font=\tiny, yshift=4pt]  {$\mathcal{O}(N_\ell^{-1/2})$};
\addplot[dashed, domain=500:15899, samples=200, thick]
{0.04*x^(-2/3)}
node[pos=0.4, below, font=\tiny, yshift=-8pt]  {$\mathcal{O}(N_\ell^{-2/3})$};

\nextgroupplot[title={$(p_t, p_\x)=(1,3)$}]
	\addplot[red, mark=*, thick]  coordinates {
(48, 0.68268107)
(132, 0.53172084)
(171, 0.40349994)
(207, 0.29807641)
(288, 0.25828213)
(360, 0.19754855)
(438, 0.12860291)
(480, 0.06166220)
(594, 0.05629058)
(702, 0.03953058)
(756, 0.02773530)
(900, 0.01491484)
(1149, 0.01155101)
(1869, 0.00697553)
(2658, 0.00448383)
(3633, 0.00314751)
(4806, 0.00229485)
(6369, 0.00138439)
(9540, 0.00095091)
    	};
\addlegendentry{ $\eta_\ell$}
	\addplot[dashed,red, mark=*, thick, forget plot]  coordinates {
(48, 0.68268107)
(384, 0.29939613)
(3072, 0.06163133)
(24576, 0.00576718)
    	};
\addplot[blue, mark=square, thick] coordinates {
(48, 0.21310357)
(132, 0.27452295)
(171, 0.25774511)
(207, 0.21599770)
(288, 0.20121130)
(360, 0.16935267)
(438, 0.11482054)
(480, 0.05902674)
(594, 0.04973477)
(702, 0.03424057)
(756, 0.02341279)
(900, 0.01324843)
(1149, 0.01072183)
(1869, 0.00641851)
(2658, 0.00406029)
(3633, 0.00273615)
(4806, 0.00196902)
(6369, 0.00122764)
(9540, 0.00084007)
    	};
\addlegendentry{$\norm{u_\ell-\widehat u_\ell}{\W}$}
\addplot[dashed, blue, mark=square, thick, forget plot] coordinates {
(48, 0.21310357)
(384, 0.21556008)
(3072, 0.05899820)
    	};    	
\addplot[green, mark=x,thick] coordinates {
(48, 1.30384451)
(132, 0.96566541)
(171, 0.69001093)
(207, 0.47617308)
(288, 0.35937472)
(360, 0.24490737)
(438, 0.13443356)
(480, 0.05788556)
(594, 0.03949648)
(702, 0.02089551)
(756, 0.01126620)
(900, 0.00866048)
(1149, 0.00543045)
(1869, 0.00294854)
(2658, 0.00176659)
(3633, 0.00129534)
(4806, 0.00069757)
(6369, 0.00040653)
(9540, 0.00028454)
};
\addlegendentry{$\norm{u-u_\ell}{L^2(\Sigma)}$}
\addplot[dashed, green, mark=x,thick, forget plot] coordinates {
(48, 1.30384451)
(384, 0.49093980)
(3072, 0.05777765)
(24576, 0.00295535)
};
\addplot[dashed, domain=1000:20000, samples=200, thick]
{60*x^(-7/6)}
node[pos=0.4, above, font=\tiny, xshift=14pt]
{$\mathcal{O}(N_\ell^{-7/6})$};
\addplot[dashed, domain=1000:20000, samples=200, thick]
{30*x^(-4/3)}
node[pos=0.3, below, font=\tiny, xshift=-14pt]
{$\mathcal{O}(N_\ell^{-4/3})$};

\end{groupplot}
\end{tikzpicture}
\caption{Error estimators and $L^2$-errors for the smooth problem with solution \eqref{problem:cos} on the square $\Omega=(0,1)^2$ over the number of degrees of freedom $N_\ell$. The solid lines represent adaptive parabolic refinement, while the dashed lines represent uniform parabolic refinement.}
\label{fig:cos}
\end{figure}
	\subsubsection{Almost singular problem on the square}
We next prescribe the exact solution as shifted heat kernel 
\begin{equation} \label{problem:heatkernel}
	U(t,\x) := G(t,\x-\z) 
\end{equation} with a fixed singularity $\z$ lying outside of $\overline\Omega$ with $\Omega = (0,1)^2$. The initial condition $U_0$ and the Neumann datum $\phi$ are chosen accordingly. Note that $U_0 = 0$. While the Dirichlet datum $u$ is smooth, depending on the choice of $\z$, significant differences between adaptive and uniform refinement can be observed. We examine the problem with the singularity $\z = (0.5,-0.25)$. In the numerical results, displayed in Figure~\ref{fig:hk}, we observe optimal rates for adaptive refinement and uniform refinement with $p_\x = 1$, while uniform refinement with $p_\x=3$ still seems to be in a pre-asymptotic regime. Again, the displayed curves indicate reliability and efficiency of the estimator $\eta_\ell$. The strong adaptive refinement towards the singularity $\z$ can be seen in Figure~\ref{fig:raumzeit-gitter}. The smallest elements are of the size $h_t = 2^{-11}$ and $h_\x= 2^{-6}$.
\begin{figure}[H]
\centering
\begin{tikzpicture}
\begin{groupplot}[
  group style={
    group size=2 by 1,
    horizontal sep=1.5cm,
  },
  legend style={
    at={(0.03,0.03)},
    anchor=south west,
    font=\tiny,
    draw=none,
    fill=none
  },
  width=0.5\textwidth,
  height=0.5\textwidth,
  xmode=log,
  ymode=log,
  xmin=9, xmax=75000,
  ymin=5e-5, ymax=0.5, 
  xmajorgrids, ymajorgrids,
  xlabel={$N_\ell$},
  yticklabel pos=left,
  ytick pos=left,
  tick label style={font=\scriptsize},
]

\nextgroupplot[title={$(p_t, p_\x)=(1,1)$}]
\addplot[red, mark=*, thick]  coordinates {
(16, 0.38392991)
(25, 0.25712474)
(34, 0.18209832)
(67, 0.10299016)
(103, 0.04980408)
(145, 0.03554675)
(201, 0.02365319)
(318, 0.01456484)
(547, 0.00866368)
(927, 0.00533149)
(1576, 0.00357682)
(2867, 0.00251542)
(5188, 0.00184429)
(9670, 0.00137746)
(18750, 0.00104602)
    	};
\addlegendentry{ $\eta_\ell$}

\addplot[dashed, red, mark=*, thick, forget plot]  coordinates {
(16, 0.38392991)
(128, 0.24986678)
(1024, 0.09653984)
(8192, 0.03018534)
(65536, 0.00348550)
    	};

\addplot[blue, mark=square, thick] coordinates {
(16, 0.05553894)
(25, 0.15289024)
(34, 0.18789011)
(67, 0.10547364)
(103, 0.04518959)
(145, 0.03638633)
(201, 0.02384827)
(318, 0.01377075)
(547, 0.00735756)
(927, 0.00442049)
(1576, 0.00284595)
(2867, 0.00188875)
(5188, 0.00132100)
(9670, 0.00093635)
    	};
\addlegendentry{$\norm{u_\ell-\widehat u_\ell}{\W}$}

\addplot[dashed, blue, mark=square, thick, forget plot] coordinates {
(16, 0.05553894)
(128, 0.14371257)
(1024, 0.12482582)
(8192, 0.03259556)
    	};

\addplot[green, mark=x, thick] coordinates {
(16, 0.36603064)
(25, 0.32951893)
(34, 0.24321804)
(67, 0.07691529)
(103, 0.02637869)
(145, 0.02302781)
(201, 0.01723118)
(318, 0.00909055)
(547, 0.00283812)
(927, 0.00134301)
(1576, 0.00070319)
(2867, 0.00039646)
(5188, 0.00023344)
(9670, 0.00013746)
(18750, 0.00008008)
	};
\addlegendentry{$\norm{u-u_\ell}{L^2(\Sigma)}$}

\addplot[dashed, green, mark=x, thick, forget plot] coordinates {
(16, 0.36603064)
(128, 0.28075646)
(1024, 0.07225686)
(8192, 0.01818744)
(65536, 0.00076225)
	};
\addplot[dashed, domain=500:27285, samples=200, thick]
{0.2*x^(-1/2)}
node[pos=0.5, above, font=\tiny, yshift=8pt]  {$\mathcal{O}(N_\ell^{-1/2})$};
\addplot[dashed, domain=500:27285, samples=200, thick]
{0.05*x^(-2/3)}
node[pos=0.3, below, font=\tiny, yshift=-8pt]  {$\mathcal{O}(N_\ell^{-2/3})$};

\nextgroupplot[title={$(p_t, p_\x)=(1,3)$}]
\addplot[red, mark=*, thick]  coordinates {
(48, 0.38459476)
(81, 0.25826810)
(114, 0.20336809)
(201, 0.13043715)
(255, 0.09692337)
(327, 0.05038598)
(453, 0.03744830)
(579, 0.02483111)
(726, 0.01863781)
(1023, 0.01220149)
(1539, 0.00673405)
(2391, 0.00396808)
(3219, 0.00266151)
(4644, 0.00171572)
(6510, 0.00115425)
(9465, 0.00078050)
(13911, 0.00053101)
    	};
\addlegendentry{ $\eta_\ell$}
\addplot[dashed, red, mark=*, thick, forget plot]  coordinates {
(48, 0.38459476)
(384, 0.24975856)
(3072, 0.09606207)
(24576, 0.02957136)
    	};
\addplot[blue, mark=square, thick] coordinates {
(48, 0.05576282)
(81, 0.15204158)
(114, 0.18236075)
(201, 0.12784996)
(255, 0.08174955)
(327, 0.04342870)
(453, 0.03609373)
(579, 0.02313102)
(726, 0.01816141)
(1023, 0.01158957)
(1539, 0.00599283)
(2391, 0.00365348)
(3219, 0.00239334)
(4644, 0.00155937)
(6510, 0.00100631)
(9465, 0.00064785)
    	};
\addlegendentry{$\norm{u_\ell-\widehat u_\ell}{\W}$}
\addplot[dashed, blue, mark=square, thick, forget plot] coordinates {
(48, 0.05576282)
(384, 0.14352626)
(3072, 0.12447568)
    	};
\addplot[green, mark=x,thick] coordinates {
(48, 0.36856109)
(81, 0.32694969)
(114, 0.23371921)
(201, 0.09201337)
(255, 0.05783275)
(327, 0.02617211)
(453, 0.02462414)
(579, 0.01980622)
(726, 0.01259602)
(1023, 0.00720335)
(1539, 0.00280029)
(2391, 0.00151991)
(3219, 0.00098282)
(4644, 0.00058043)
(6510, 0.00037559)
(9465, 0.00024791)
(13911, 0.00012978)
};
\addlegendentry{$\norm{u-u_\ell}{L^2(\Sigma)}$}
\addplot[dashed, green, mark=x,thick, forget plot] coordinates {
(48, 0.36856109)
(384, 0.28089703)
(3072, 0.07218230)
(24576, 0.01817997)
};
\addplot[dashed, domain=1000:24000, samples=200, thick]
{53*x^(-7/6)}
node[pos=0.5, above, font=\tiny, xshift=15pt]
{$\mathcal{O}(N_\ell^{-7/6})$};
\addplot[dashed, domain=1000:24000, samples=200, thick]
{25*x^(-4/3)}
node[pos=0.3, below, font=\tiny, xshift=-15pt]
{$\mathcal{O}(N_\ell^{-4/3})$};

\end{groupplot}
\end{tikzpicture}
\caption{Error estimators and $L^2$-errors for the smooth problem with solution \eqref{problem:heatkernel} and $\z=(0.5,-0.25)$ on the square $\Omega =(0,1)^2$ over the number of degrees of freedom $N_\ell$. The solid lines represent adaptive parabolic refinement, while the dashed lines represent uniform parabolic refinement.}
\label{fig:hk}
\end{figure}
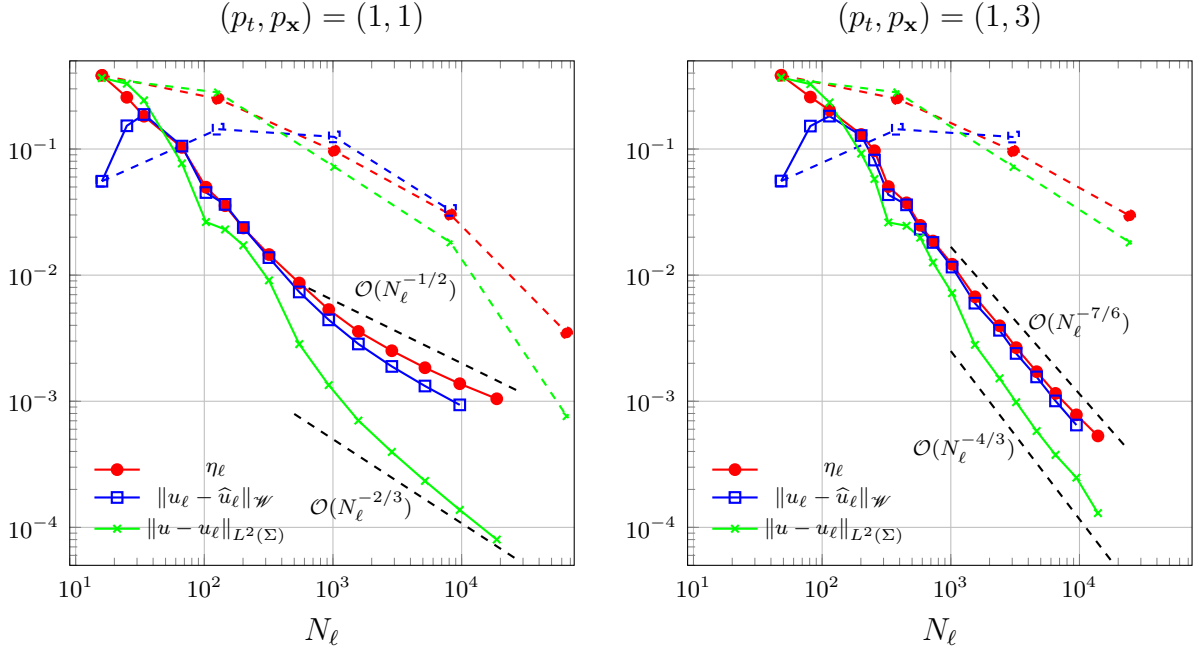

\begin{figure}[htbp]
    \centering
    \includegraphics[width=0.5\textwidth]{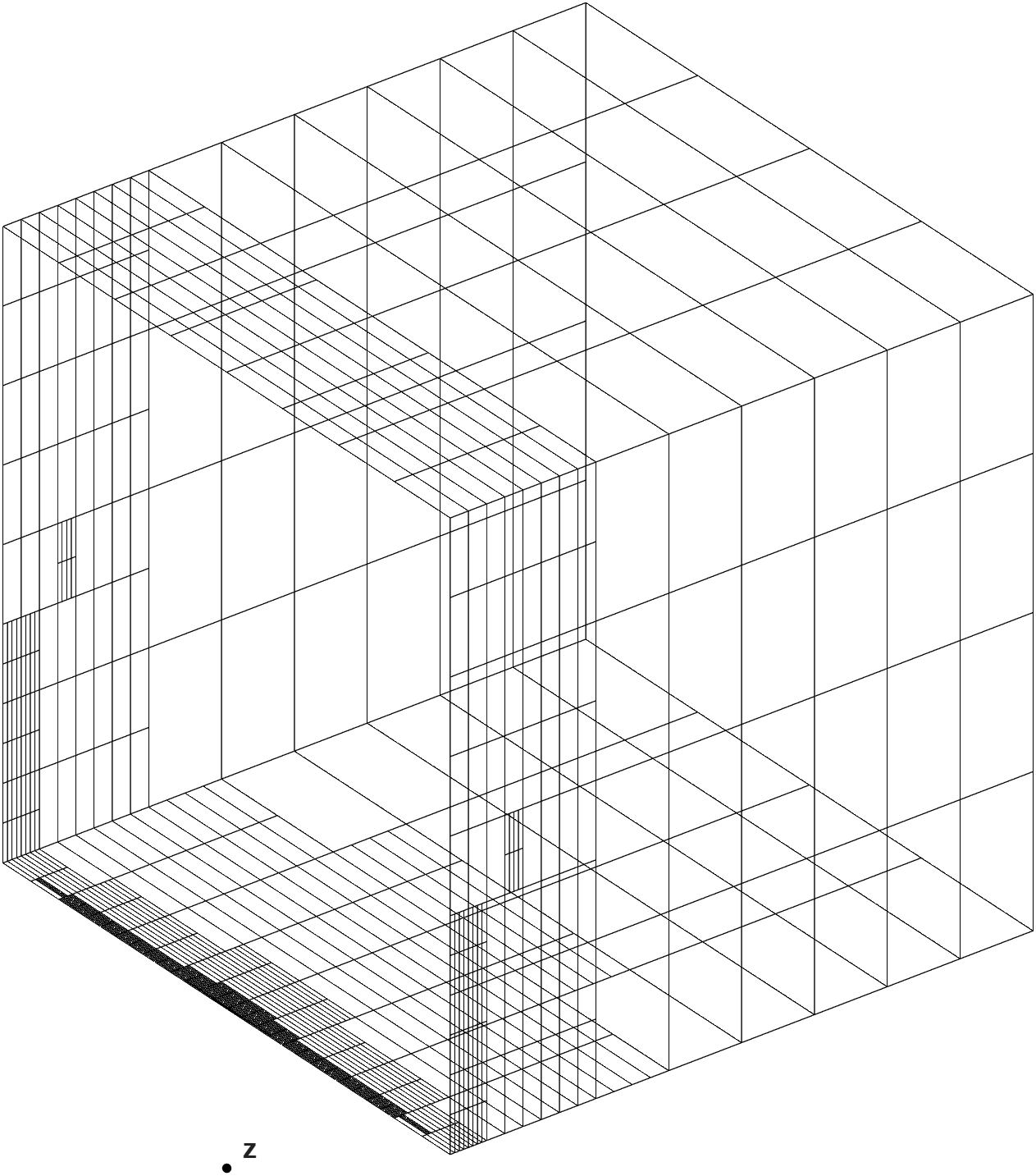}
    \caption{Adaptive mesh $\PP_9$ for the problem with solution~\eqref{problem:heatkernel} and $(p_t,p_\x)=(1,1)$. The mesh consists of $1024$ prisms with $N_9=927$ degrees of freedom. The singularity $\z$ is depicted in space and time relative to the mesh. }
    \label{fig:raumzeit-gitter}
\end{figure}

\subsubsection{Singular problem on the square}

We consider the heat equation \eqref{eq:interior} with given data
\begin{align}\label{problem:singular}
U_0(\x)=0\quad\text{and}\quad \phi(t,\x)=1
\end{align}
on the square $\Omega = (0,1)^2$.
Due to the incompatibility between $\phi$ and $U_0$ in the sense that $\phi(0,\x)  \neq \partial_\n U_0(\x)$, we expect the unknown solution $U$ (and its  Dirichlet trace $u$) to be singular at $t = 0$. In the numerical results, displayed in Figure \ref{fig:square}, we observe convergence rates of order $\mathcal O(N_\ell^{-1/2})$ for uniform refinement, which is suboptimal for $p_\x \neq 1$. In contrast, for $p_\x = 3$, adaptive refinement yields nearly optimal rates, 
namely approximately $\mathcal{O}(N_\ell^{-1})$, in contrast to
$\mathcal{O}(N_\ell^{-7/6})$ expected for smooth solutions.
\begin{figure}[H]
\centering
\begin{tikzpicture}
\begin{groupplot}[
  group style={
    group size=2 by 1,
    horizontal sep=1.5cm,
  },
  legend style={
    at={(0.03,0.03)},
    anchor=south west,
    font=\tiny,
    draw=none,
    fill=none
  },
  width=0.5\textwidth,
  height=0.5\textwidth,
  xmode=log,
  ymode=log,
  xmin=9, xmax=75000,
  ymin=6*1e-4, ymax=0.2, 
  xmajorgrids, ymajorgrids,
  xlabel={$N_\ell$},
  yticklabel pos=left,
  ytick pos=left,
  tick label style={font=\scriptsize},
]

\nextgroupplot[title={$(p_t, p_\x)=(1,1)$}]
\addplot[red, mark=*, thick]  coordinates {
(16, 0.16805487)
(46, 0.13384419)
(101, 0.08779803)
(146, 0.06376847)
(215, 0.04757749)
(571, 0.03786302)
(979, 0.02672983)
(1171, 0.01931286)
(2611, 0.01530933)
(5818, 0.01083177)
(7705, 0.00789047)
(10188, 0.00565068)
(34775, 0.00449432)
    	};
\addlegendentry{ $\eta_\ell$}
\addplot[dashed, red, mark=*, thick, forget plot] coordinates {
(16, 0.16805487)
(128, 0.07001447)
(1024, 0.02575517)
(8192, 0.00932416)
(65536, 0.00342777)
    	};
\addplot[blue, mark=square, thick] coordinates {
(16, 0.10800476)
(46, 0.09719139)
(101, 0.06345296)
(146, 0.04538091)
(215, 0.03463146)
(571, 0.02756358)
(979, 0.01896144)
(1171, 0.01422723)
(2611, 0.01091313)
(5818, 0.00794411)
(7705, 0.00579460)
(10188, 0.00407377)
    	};
\addlegendentry{$\norm{u_\ell-\widehat u_\ell}{\W}$}
\addplot[dashed, blue, mark=square, thick, forget plot] coordinates {
(16, 0.10800476)
(128, 0.04992192)
(1024, 0.01804172)
(8192, 0.00641649)
    	};
\addplot[dashed, domain=500:65000, samples=200, thick]
{1.3*x^(-1/2)}
node[pos=0.5, above, font=\tiny, yshift=10pt]  {$\mathcal{O}(N_\ell^{-1/2})$};

\nextgroupplot[title={$(p_t, p_\x)=(1,3)$}]
	\addplot[red, mark=*, thick]  coordinates {
(48, 0.13322149)
(132, 0.10082559)
(171, 0.08160504)
(207, 0.05739726)
(306, 0.04617677)
(396, 0.03427868)
(480, 0.02113466)
(708, 0.01801018)
(909, 0.01417663)
(1035, 0.00951298)
(1485, 0.00728509)
(2019, 0.00571176)
(2337, 0.00432772)
(3492, 0.00329625)
(5184, 0.00251844)
(6978, 0.00179879)
(12342, 0.00148197)
    	};
\addlegendentry{ $\eta_\ell$}

\addplot[dashed, red, mark=*, thick, forget plot] coordinates {	
(48, 0.13322149)
(384, 0.05680204)
(3072, 0.02079710)
(24576, 0.00746973)
    	};
\addplot[blue, mark=square, thick] coordinates {
(48, 0.07150008)
(132, 0.05937406)
(171, 0.05151141)
(207, 0.03899098)
(306, 0.02921781)
(396, 0.02192885)
(480, 0.01400797)
(708, 0.01166527)
(909, 0.00957810)
(1035, 0.00598778)
(1485, 0.00459763)
(2019, 0.00356283)
(2337, 0.00258841)
(3492, 0.00213103)
(5184, 0.00162093)
(6978, 0.00082062)
    	};
\addlegendentry{$\norm{u_\ell-\widehat u_\ell}{\W}$}
\addplot[dashed, blue, mark=square, thick, forget plot] coordinates {
(48, 0.07150008)
(384, 0.03896287)
(3072, 0.01399774)
    	};
\addplot[dashed, domain=500:24000, samples=200, thick]
{1.5*x^(-1/2)}
node[pos=0.5, above, font=\tiny, yshift=10pt] {$\mathcal{O}(N_\ell^{-1/2})$};
\addplot[dashed, domain=1000:7000, samples=200, thick]
{4.5*(x)^(-1)}
node[pos=0.4, below, font=\tiny, xshift=-14pt]
{$\mathcal{O}(N_\ell^{-1})$};

\end{groupplot}
\end{tikzpicture}
\caption{Error estimators for the singular problem with data \eqref{problem:singular} on the square $\Omega= (0,1)^2$ over the number of degrees of freedom $N_\ell$. The solid lines represent adaptive parabolic refinement, while the dashed lines represent uniform parabolic refinement.}
\label{fig:square}
\end{figure}

\subsubsection{Singular problem on the L-shape}
We consider the heat equation \eqref{eq:interior} with given data \eqref{problem:singular}
on the L-shaped domain $\Omega = (-1,1)^2\setminus [0,1]^2$.
Due to the incompatibility between $\phi$ and $U_0$ in the sense that $\phi(0,\x)  \neq \partial_\n U_0(\x)$, we expect the unknown solution $U$ (and its  Dirichlet trace $u$) to be singular at $t = 0$. 
In the numerical results, displayed in Figure \ref{fig:lshape}, we observe convergence rates of order $\mathcal O(N_\ell^{-1/2})$ for uniform refinement, which is suboptimal for $p_\x \neq 1$. 
In contrast, for $p_\x = 3$, adaptive refinement yields nearly optimal rates, 
namely approximately $\mathcal{O}(N_\ell^{-1})$, in contrast to 
$\mathcal{O}(N_\ell^{-7/6})$ expected for smooth solutions.

\begin{figure}[H]
\centering
\begin{tikzpicture}
\begin{groupplot}[
  group style={
    group size=2 by 1,
    horizontal sep=1.5cm,
  },
  legend style={
    at={(0.03,0.03)},
    anchor=south west,
    font=\tiny,
    draw=none,
    fill=none
  },
  width=0.5\textwidth,
  height=0.5\textwidth,
  xmode=log,
  ymode=log,
  xmin=25, xmax=22000,
  ymin=2.5e-3, ymax=0.3, 
  xmajorgrids, ymajorgrids,
  xlabel={$N_\ell$},
  yticklabel pos=left,
  ytick pos=left,
  tick label style={font=\scriptsize},
]

\nextgroupplot[title={$(p_t, p_\x)=(1,1)$}]
\addplot[red, mark=*, thick]  coordinates {
(32, 0.22326481)
(65, 0.18597068)
(117, 0.12060921)
(191, 0.09299208)
(307, 0.06993710)
(441, 0.04842723)
(911, 0.03686948)
(1332, 0.02712761)
(2242, 0.01854598)
(4103, 0.01420694)
(7619, 0.01072142)
(13517, 0.00753602)
(20445, 0.00576630)
    	};
\addlegendentry{ $\eta_\ell$}
\addplot[dashed, red, mark=*, thick, forget plot] coordinates {
(32, 0.22326481)
(256, 0.08967364)
(2048, 0.03293353)
(16384, 0.01196692)
    	};
\addplot[blue, mark=square, thick] coordinates {
(32, 0.15154755)
(65, 0.13688088)
(117, 0.08261737)
(191, 0.06517136)
(307, 0.04959320)
(441, 0.03459898)
(911, 0.02592335)
(1332, 0.01915052)
(2242, 0.01340915)
(4103, 0.01012794)
(7619, 0.00723153)
(13517, 0.00512112)
  	};
\addlegendentry{$\norm{u_\ell-\widehat u_\ell}{\W}$}
\addplot[dashed, blue, mark=square, thick, forget plot] coordinates {
(32, 0.15154755)
(256, 0.06184246)
(2048, 0.02250841)
(16384, 0.00805016)
  	};
\addplot[dashed, domain=500:20000, samples=200, thick]
{2.4*x^(-1/2)}
node[pos=0.5, above, font=\tiny, yshift=10pt]  {$\mathcal{O}(N_\ell^{-1/2})$};

\nextgroupplot[title={$(p_t, p_\x)=(1,3)$}]
	\addplot[red, mark=*, thick]   coordinates {
(96, 0.20369195)
(195, 0.15885415)
(267, 0.11556105)
(363, 0.08756373)
(549, 0.06893671)
(759, 0.05676627)
(900, 0.03852397)
(1140, 0.02858015)
(1650, 0.02280851)
(1971, 0.01712713)
(2394, 0.01211432)
(3777, 0.00935571)
(4683, 0.00692434)
(6798, 0.00520032)
(10680, 0.00377282)
    	};
\addlegendentry{ $\eta_\ell$}
\addplot[dashed, red, mark=*, thick, forget plot] coordinates {	
(96, 0.20369195)
(768, 0.08152562)
(6144, 0.02964808)
    	};
\addplot[blue, mark=square, thick] coordinates {
(96, 0.13121095)
(195, 0.10213380)
(267, 0.06534443)
(363, 0.05733027)
(549, 0.04317101)
(759, 0.03738809)
(900, 0.02578542)
(1140, 0.01870709)
(1650, 0.01534566)
(1971, 0.01172180)
(2394, 0.00833007)
(3777, 0.00645141)
(4683, 0.00493578)
(6798, 0.00337501)
    	};
\addlegendentry{$\norm{u_\ell-\widehat u_\ell}{\W}$}
\addplot[dashed, blue, mark=square, thick, forget plot] coordinates {	
(96, 0.13121095)
(768, 0.05453275)
(6144, 0.01971571)
    	};
\addplot[dashed, domain=500:6200, samples=200, thick]
{3*x^(-1/2)}
node[pos=0.5, above, font=\tiny, yshift=8pt] {$\mathcal{O}(N_\ell^{-1/2})$};
\addplot[dashed, thick, domain=900:6500, samples=200] {17*x^(-1)}
node[pos=0.5, below, font=\tiny, xshift=-14pt]
{$\mathcal{O}(N_\ell^{-1})$};
\end{groupplot}
\end{tikzpicture}
\caption{Error estimators for the singular problem with data \eqref{problem:singular} on the L-shape $\Omega= (-1,1)^2\setminus [0,1]^2$ over the number of degrees of freedom $N_\ell$.  The solid lines represent adaptive parabolic refinement, while the dashed lines represent uniform parabolic refinement.}
\label{fig:lshape}
\end{figure}
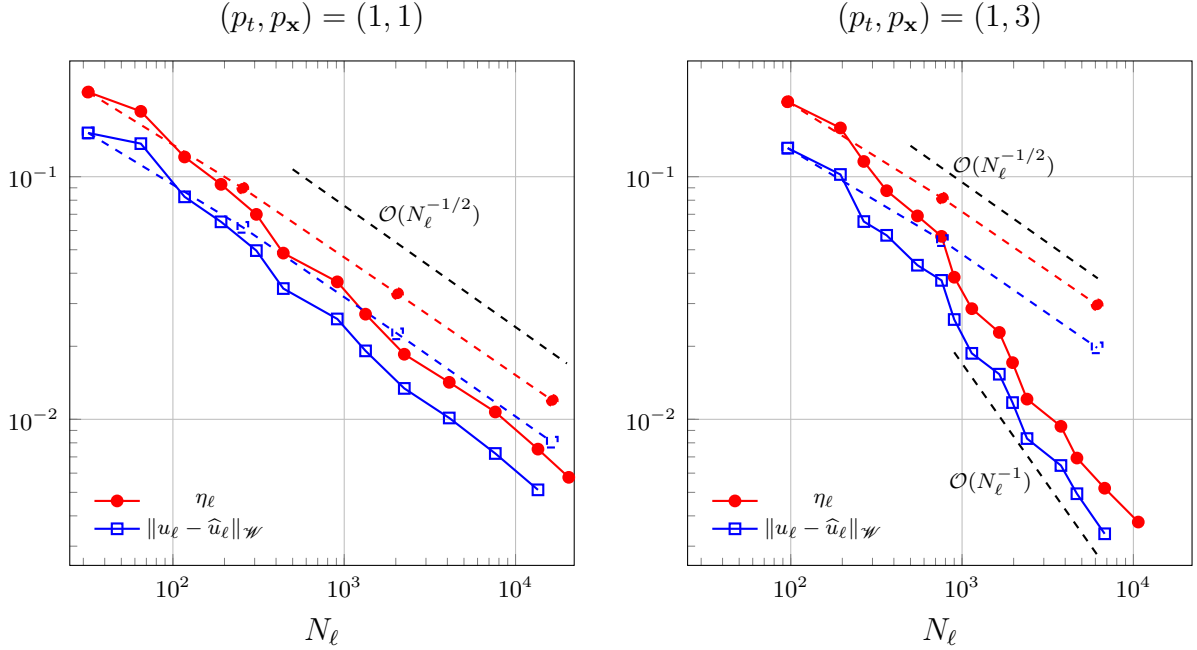

\appendix

\section{Proof of Proposition~\ref{prop:varphi}}\label{appendix:Prop}

In this section, we prove Proposition~\ref{prop:varphi}. 
We start with the following lemma.

\begin{lemma}\label{lem:partitionofunity}
 Let $v_h \in S^{1,1}(\PP_h)$. Then, $v_h(\z) = 0$ for all free nodes $\z \in \NN_h$ implies that $v_h = 0$.
\end{lemma}
\begin{proof} 
Since  the set of all nodes $\NN_h \cup \NN_h^\perp$ is finite, there exists a linear functional $L\colon \mathbb R^{d} \to \mathbb R$
which is injective on the set of vertices. Indeed, $L$ can be chosen as $L(\x):=\x\cdot \mathbf{v}$ for any 
$\mathbf{v}\in\R^d$ that does not lie in the orthogonal complement of any line $\text{span}\{\z - \z'\}$, $\z,\z' \in \NN_h \cup \NN_h^\perp$. 
Assume, by contradiction, that there exists a hanging node $\z$ such that $v_h(\z)\neq 0$, i.e.,
\[M := \max_{\z\in \NN_h^\perp} |v_h(\z)| >0 .\]
Choose $\z^*\in  \NN_h^\perp$ such that $|v_h(\z^*)| = M$
and, among all nodes with this property, such that $L(\z^*)$ is maximal.

By definition of a hanging node, there exists a prism $P\in \PP_h$
such that $\z^* \in P$, but $\z^*$ is not a vertex of $P$.
Let $F$ be the edge or face of $P$ such that $\z^*$ lies in the
interior of $F$. If $F$ is a face, then it is either a triangle or a rectangle.
Let $\z_1,\dots,\z_m$ be the vertices of $F$. Since
$\z^*$ lies in the interior of $F$, there exist nonnegative weights
$w_1,\dots,w_m <1$ with
$\sum_{i=1}^m w_i = 1$ such that $\z^* = \sum_{i=1}^m w_i \z_i$.

Moreover, since $v_h|_P$ is a polynomial of degree one in space and time, the
restriction of $v_h$ to $F$ is given by interpolation of its values at the
vertices of $F$. More precisely, it is affine on edges and triangular faces, and bi-affine on rectangular faces. Therefore,
\[ v_h(\z^*) = \sum_{i=1}^m w_i v_h(\z_i).\]
 By the maximality of $M$, we have that
\[    M=|v_h(\z^*)|\leq\sum_{i=1}^m w_i|v_h(\z_i)|
\leq \sum_{i=1}^m w_i M = M.\]
It follows that $|v_h( \z_i)| = M$ for all $i\in \{1,\ldots,m\}$.

On the other hand, applying the linear functional $L$ to the convex representation of
$\z^*$, we obtain that
\[L(\z^*) = \sum_{i=1}^m w_i L(\z_i).\]
Since $L$ is injective on $\NN_h \cup \NN_h^\perp$, there exists a unique $j = \argmax_{i=1,\dots,m} L(\z_i)$.
Since the weights are nonnegative, and $\z^*$ is not one of the vertices
$\z_i$, this yields that $L(\z_j) > L(\z^*)$.

In particular, $|v_h(\z_j)| = M$ and $L(\z_j) > L(\z^*)$.
This contradicts the choice of $\z^*$ as a node with maximal value of
$L$ among all nodes at which $|v_h|$ attains the value $M$.
Hence our assumption was false, and therefore  $v_h(\z) = 0$ for all vertices $\z$. This concludes the proof of the lemma.
\end{proof}

\begin{proof}[Proof of Proposition~\ref{prop:varphi}]
We only consider the case $d=2$. The case $d=3$ can be proven similarly but is more technical. 
We also mention that the proof simplifies for both $d=2$ and $d=3$ if only meshes generated by the refinement strategy from Section~\ref{sec:adaptiveAlg} are considered.
The proof is split into two steps. 

\textbf{Step 1:} For every $\z\in \NN_h$, we recursively define the patch $\pi_h^i(\z)$ for $i \in \N$ by
\[\pi^1_h(\z):= \bigcup \set{P\in \PP_h}{\z\in P},\quad \text{and}\quad\pi^i_h(\z):=\bigcup \set{P\in \PP_h}{P\cap \pi^{i-1}_h(\z)\neq \emptyset}\quad\text{for } i\ge 2.
\]
In the following, we construct a suitable $\varphi_{h,\z}$ and show the inclusion
\begin{align}\label{eq:prophelp1}
\omega_h(\z):=\supp(\varphi_{h,\z}) \subseteq \pi_{h}^3(\z)
\end{align}
for all $\z \in \NN_h$.
Together with Lemma~\ref{lem:partitionofunity}, this immediately proves (i)--(iii).

\textbf{Step 2:} We only consider interior nodes $\z \in \NN_h$, i.e., $\z \in (0,T) \times \Gamma$. Boundary nodes can be treated analogously.
 We denote by $P_1, P_2, P_3, P_4\in \PP_h$  the prisms adjacent to $\z$, i.e., the prisms in the patch $\Pi^1_{h}(\z)$.
Let $e_1, e_2, e_3, e_4$ be the longest edges of the prisms adjacent to $\z$ in the sense that for all $i,j\in\{1,2,3,4\}$ either $e_i\cap P_j=\z$ or $e_i\cap P_j$ is an edge of $P_j$.  
Let $e_1$ and $e_2$ be edges of type $J\times \{\x\}$ and $e_3$ and $e_4$ edges of type $\{t\}\times K$. 
Define $\omega_h'(\z)$ as the union of all prisms $P \in \mathcal{P}_h$ such that $P \cap e_i$ is an edge of $P$ for at least one $i \in \{1,2,3,4\}$; see Figure~\ref{fig:prop2d1} for an example. By construction, it follows immediately that $\omega_h'(\z)\subseteq \pi^2_h(\z)$.
\begin{figure}[H]
\begin{tikzpicture}
    \node at (-2.75,0) {$\x$};
    \node at (0,-2.5) {$t$};
    \draw[thick] (0,0) rectangle (3,2);
    \draw[thick] (-2,0) rectangle (0,2);
    \draw[thick] (-2,-2) rectangle (-1,0);
    \draw[thick] (-1,-1) rectangle (0,0);
    \draw[thick] (0,-1.5) rectangle (0.5,0);
    \draw[thick] (0.5,-1) rectangle (1,0);
    \draw[thick] (1,-0.5) rectangle (1.5,0);
    \draw[thick] (1.5,-0.5) rectangle (2,0);  
    \draw[thick] (2,-1.25) rectangle (2.5,0);
    \draw[thick] (2.5,-0.5) rectangle (3,0); 

    \draw[ultra thick,blue!65!black] (0,0) -- (0,2)
        node[pos=0.55,right] {$e_3$};
    \draw[ultra thick,blue!65!black] (0,0) -- (3,0)
        node[pos=0.55,above] {$e_2$};
    \draw[ultra thick,blue!65!black] (0,0) -- (-2,0)
        node[pos=0.55,above] {$e_1$};
    \draw[ultra thick,blue!65!black] (0,0) -- (0,-1.5)
        node[pos=0.55,right, inner sep=0.5pt] {$e_4$};
        
	\draw[ultra thick,red!70!black] (-1,0) -- (-1,-2)
    		node[pos=0.55,left] {$e_5$};
	\draw[ultra thick,red!70!black] (0.5,0) -- (0.5,-1.5)
    		node[pos=0.55,right, inner sep=0.5pt] {$e_6$};
	\draw[ultra thick,red!70!black] (1,0) -- (1,-1)
    		node[pos=0.28,right, inner sep=0.5pt] {$e_7$};
	\draw[ultra thick,red!70!black] (1.5,0) -- (1.5,-0.5)
    		node[pos=0.55,right, inner sep=0.5pt] {$e_8$};
	\draw[ultra thick,red!70!black] (2,0) -- (2,-1.25)
    		node[pos=0.55,right, inner sep=0.5pt] {$e_9$};
	\draw[ultra thick,red!70!black] (2.5,0) -- (2.5,-1.25)
    		node[pos=0.55,right, inner sep=0.5pt] {$e_{10}$};

    \fill (0,0) circle (2.5pt);

    \node[above right] at (0,0) {$\z$};
    \node at (-1,1) {$P_1$};
    \node at (1.5,1) {$P_2$};
    \node at (-0.5,-0.5) {$P_3$};
    \node at (0.25,-0.25) {$P_4$};
\end{tikzpicture}
\caption{Example of $\omega_h'(\z)$ for $d=2$.}
\label{fig:prop2d1}
\end{figure}
Moreover, for the edges $e_3$ and $e_4$, all prisms in $\omega_h'(\z)$ lying on the same side of $e_3$ or $e_4$ have the same (temporal) length by Assumption~(ii) of Section~\ref{sec:meshes} and entirely lie on $e_3$ or $e_4$ by Assumption~(i) of Section~\ref{sec:meshes}. Prisms on one side of $e_1$ or $e_2$ do not necessarily have the same (spatial) length. We define $e_5,\ldots, e_m$ as the longest edges with one endpoint in the interior of $e_1$ or $e_2$. All hanging nodes (with respect to the restriction to $\omega_h'(\z)$) on the boundary of $\omega_h'(\z)$ lie in the interior of an edge $e_i$ for $i\in \{3, 4, 5,\dots, m\}$ by the construction of $\omega_h'(\z)$ and Assumptions~(i)--(ii) of Section~\ref{sec:meshes}.  Define $\omega_h(\z)$ as the union of $\omega_h'(\z)$ and all prisms $P \in \mathcal{P}_h$ such that $P \cap e_i$ is an edge of $P$ for at least one $i \in \{3, 4, 5,\ldots, m\}$; see Figure~\ref{fig:prop2d2} for an example. By construction, it follows immediately that $\omega_h(\z)\subseteq \pi^3_h(\z)$.

 If we restrict $\PP_h$ to the prisms in $\omega_h(\z)$, all hanging nodes (with respect to the restriction to $\omega_h(\z)$) on the boundary of $\omega_h(\z)$ lie in the interior of an edge of the type $\{t'\}\times K$ by the construction of $\omega_h(\z)$ and Assumptions~(i)--(ii) of Section~\ref{sec:meshes}. The only way hanging nodes can occur on the boundary of $\omega_h(\z)$ is the configuration shown in Figure~\ref{fig:prop2d3}: two prisms in $\omega_h'(\z)$ with spatial parts $K_1, K_2$ lie next to each other and between two prisms with different spatial parts $K_3$ and $K_4$, such that $K_i\subsetneq K_3$ and $K_i\subsetneq K_4$ for $i\in \{1,2\}$. Here, either $K_1 = K_2$ or $K_1\neq K_2$ may occur, as shown in Figure~\ref{fig:prop2d3}. The extension to $\omega_h(\z)$ may then lead to new hanging nodes on the boundary of $\omega_h(\z)$, as shown on the right in Figure~\ref{fig:prop2d3}.
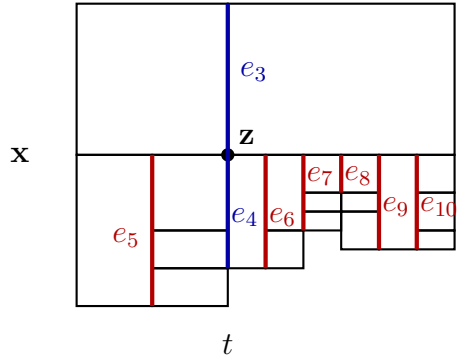
\begin{figure}[H]
\centering
\begin{tikzpicture}
    \node at (-2.75,0) {$\x$};
    \node at (0,-2.5) {$t$};
    \draw[thick] (0,0) rectangle (3,2);
    \draw[thick] (-2,0) rectangle (0,2);
    \draw[thick] (-2,-2) rectangle (-1,0);
    \draw[thick] (-1,-1) rectangle (0,0);
    \draw[thick] (0,-1.5) rectangle (0.5,0);
    \draw[thick] (0.5,-1) rectangle (1,0);
    \draw[thick] (1,-0.5) rectangle (1.5,0);
    \draw[thick] (1.5,-0.5) rectangle (2,0);  
    \draw[thick] (2,-1.25) rectangle (2.5,0);
    \draw[thick] (2.5,-0.5) rectangle (3,0);

    \draw[thick] (-1,-1.5) rectangle (0,-1);
    \draw[thick] (-1,-2) rectangle (0,-1.5);    
    \draw[thick] (1,-1.5) rectangle (0.5,-1);  
    \draw[thick] (1.5,-1) rectangle (1,-0.75);  
    \draw[thick] (1.5,-0.75) rectangle (1,-0.5);    
    \draw[thick] (1.5,-0.75) rectangle (2,-0.5);   
    \draw[thick] (1.5,-1.25) rectangle (2,-0.75); 
    \draw[thick] (2.5,-1) rectangle (3,-0.5);    
    \draw[thick] (2.5,-1.25) rectangle (3,-1);
    \fill (0,0) circle (2.5pt);
    
    \draw[ultra thick,blue!65!black] (0,0) -- (0,2)
        node[pos=0.55,right] {$e_3$};
    \draw[ultra thick,blue!65!black] (0,0) -- (0,-1.5)
        node[pos=0.55,right, inner sep=0.5pt] {$e_4$};
        
	\draw[ultra thick,red!70!black] (-1,0) -- (-1,-2)
    		node[pos=0.55,left] {$e_5$};
	\draw[ultra thick,red!70!black] (0.5,0) -- (0.5,-1.5)
    		node[pos=0.55,right, inner sep=0.5pt] {$e_6$};
	\draw[ultra thick,red!70!black] (1,0) -- (1,-1)
    		node[pos=0.28,right, inner sep=0.5pt] {$e_7$};
	\draw[ultra thick,red!70!black] (1.5,0) -- (1.5,-0.5)
    		node[pos=0.55,right, inner sep=0.5pt] {$e_8$};
	\draw[ultra thick,red!70!black] (2,0) -- (2,-1.25)
    		node[pos=0.55,right, inner sep=0.5pt] {$e_9$};
	\draw[ultra thick,red!70!black] (2.5,0) -- (2.5,-1.25)
    		node[pos=0.55,right, inner sep=0.5pt] {$e_{10}$};

    \node[above right] at (0,0) {$\z$};
\end{tikzpicture}
\caption{Example of $\omega_h(\z)$ originating from $\omega_h'(\z)$ in Figure~\ref{fig:prop2d1} for $d=2$.}
\label{fig:prop2d2}
\end{figure}

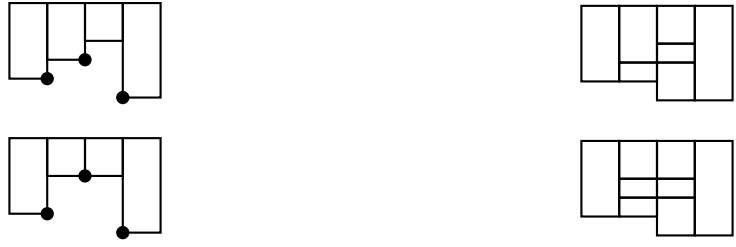
\begin{figure}[H]
\centering
\begin{minipage}{0.45\textwidth}
\centering
\begin{tikzpicture}
    \draw[thick] (0,-1) rectangle (0.5,0);
    \draw[thick] (0.5,-0.75) rectangle (1,0);
    \draw[thick] (1,-0.5) rectangle (1.5,0);  
    \draw[thick] (1.5,-1.25) rectangle (2,0);

    \fill (0.5,-1) circle (2.5pt);    
    \fill (1,-0.75) circle (2.5pt);
    \fill (1.5,-1.25) circle (2.5pt);
\end{tikzpicture}
\end{minipage}
\vspace{4mm}
\begin{minipage}{0.45\textwidth}
\centering
\begin{tikzpicture}
    \draw[thick] (0,-1) rectangle (0.5,0);
    \draw[thick] (0.5,-0.75) rectangle (1,0);
    \draw[thick] (1,-0.5) rectangle (1.5,0);  
    \draw[thick] (1.5,-1.25) rectangle (2,0);
    
    \draw[thick] (0.5,-1) rectangle (1,-0.75);
    \draw[thick] (1,-0.75) rectangle (1.5,-0.5);
    \draw[thick] (1,-1.25) rectangle (1.5,-0.75);
    
\end{tikzpicture}
\end{minipage}
\begin{minipage}{0.45\textwidth}
\centering
\begin{tikzpicture}
    \draw[thick] (0,-1) rectangle (0.5,0);
    \draw[thick] (0.5,-0.5) rectangle (1,0);
    \draw[thick] (1,-0.5) rectangle (1.5,0);  
    \draw[thick] (1.5,-1.25) rectangle (2,0);
    
    \fill (0.5,-1) circle (2.5pt);    
    \fill (1,-0.5) circle (2.5pt);
    \fill (1.5,-1.25) circle (2.5pt);
\end{tikzpicture}
\end{minipage}
\begin{minipage}{0.45\textwidth}
\centering
\begin{tikzpicture}
    \draw[thick] (0,-1) rectangle (0.5,0);
    \draw[thick] (0.5,-0.5) rectangle (1,0);
    \draw[thick] (1,-0.5) rectangle (1.5,0);  
    \draw[thick] (1.5,-1.25) rectangle (2,0);
    
    \draw[thick] (0.5,-0.75) rectangle (1,-0.5);
    \draw[thick] (0.5,-1) rectangle (1,-0.75);
    \draw[thick] (1,-0.75) rectangle (1.5,-0.5);
    \draw[thick] (1,-1.25) rectangle (1.5,-0.75);
    
\end{tikzpicture}
\end{minipage}
\caption{Local configuration illustrating how new hanging nodes can arise when extending $\omega_h'(\z)$ (left) to $\omega_h(\z)$ (right) for $d=2$.}
\label{fig:prop2d3}
\end{figure}

 Since any polynomial of degree one in time and space on a prism $P = J\times K$ is uniquely determined by its values at the vertices of $P$, we may define $\varphi_{h,\mathbf{z}} \in S^{1,1}(\mathcal{P}_h)$ iteratively as follows:

\begin{itemize}
\item Let $\z'\in \omega_h'(\z)$ be a vertex of a prism of $\PP_h$. If $\z'\in e_i$ for some $i\in\{1,2,3,4\}$, then  $\varphi_{h,\z}(\z')$ is obtained by linear interpolation of $\varphi_{h,\z}(\z)= 1$ and $\varphi_{h,\z}(\z_i)= 0$ with the other endpoint $\z_i\neq \z$ of $e_i$.  If $\z'\in e_i = \text{conv}\{\z_i,\z_i'\}$ for $i\in\{5,\ldots, m\}$ with $\z_i'$ lying on $e_1$ or $e_2$, then  $\varphi_{h,\z}(\z')$ is obtained by linear interpolation of the value of $\varphi_{h,\z}(\z_i')$  and $\varphi_{h,\z}(\z_i)= 0$.
In the configuration of Figure~\ref{fig:prop2d3}, the $\z_i$ are highlighted by a dot.
 Otherwise, choose $\varphi_{h,\z}(\z')=0$.
\item Let $\z'\in \omega_h(\z) \setminus \omega_h'(\z)$ be a vertex of a prism of $\PP_h$. If $\z'$ is not a hanging node (with respect to the restriction to $\omega_h(\z)$) on the boundary of $\omega_h(\z)$, then $\varphi_{h,\z}(\z') = 0$. 
Since $\varphi_{h,\z}(\z_i) = 0$, we can also choose $\varphi_{h,\z}(\z') = 0$ if $\z'$ is a hanging node on the boundary of $\omega_h(\z)$, as $\varphi_{h,\z}$ on newly introduced outer edges from $\omega_h'(\z)$ to $\omega_h(\z)$ may be chosen as zero; cf. Figure~\ref{fig:prop2d3}.
\end{itemize}
This proves \eqref{eq:prophelp1}. Moreover, the constructed function $\varphi_{h,\z}$ is nonnegative at all vertices of $\PP_h$. Since it is multilinear on each prism and thus also on $\Sigma$, this implies $\varphi_{h,\z}(t,\x)\ge 0$ for all $(t,\x)\in \Sigma$.

\end{proof}

\section{Numerical computation}\label{sec:numericalcomp}

This section describes the implementation of the Galerkin system \eqref{eq:Galerkin} with $f$ given as in \eqref{eq:direct} and of the estimator \eqref{eq:estimator} for $d\in\{2,3\}$. Throughout, let $\PP_h$ be a prismatic mesh of $\Sigma$.

\subsection{Right-hand side}

We discuss the discretization of the two parts $ (1/2-\mathscr{N})\phi_h$ and $\M_1 U_0$ of the right-hand side $f$ given as in \eqref{eq:direct}.
\subsubsection{Neumann datum}
For our computations, we replace the Neumann datum $\phi$ by its $L^2(\Sigma)$-orthogonal projection $\phi_h := \Pi_h \phi$ onto $\PP_h$-piecewise polynomials of degree $p_t$ in time and $p_\x$ in space.
In particular, we replace the right-hand side $f$ of \eqref{eq:direct} by 
$$ f_h := (1/2-\mathscr{N})\phi_h - \M_1 U_0.$$
To ensure that the best possible convergence rate $\OO\big(N_h^{-\frac{\min\{p_\x+1/2,2p_t+3/2\}}{d+1}}\big)$ for parabolically scaled meshes $\PP_h$ (see \eqref{eq:rates}) is not deteriorated,
 we must control the resulting error $\norm{ u_h - \widetilde u_h}{H^{1/2,1/4}(\Sigma)}$, where $\widetilde u_h$ denotes the Galerkin solution  of \eqref{eq:direct} with $f$ replaced by $f_h$. By stability of the Galerkin projection and the operator  $\mathscr{N}$, we see that 

$$ \norm{u_h - \widetilde u_h}{H^{1/2,1/4}(\Sigma)} \lesssim \norm{(1/2-\mathscr{N})(\phi-\phi_h)}{H^{-1/2,-1/4}(\Sigma)}
\lesssim \norm{\phi-\phi_h}{H^{-1/2,-1/4}(\Sigma)}.$$
Using the characterization of the dual norm and the $L^2(\Sigma)$-orthogonality of $\Pi_h$, we get that
\begin{align*}
\norm{\phi-\phi_h}{H^{-1/2,-1/4}(\Sigma)}
&= \sup_{v\in H^{1/2,1/4}(\Sigma)\setminus\{0\}}
\frac{\bigl|\langle (1-\Pi_h)\phi, v\rangle_{\Sigma}\bigr|}{\|v\|_{H^{1/2,1/4}(\Sigma)}}\\
&= \sup_{v\in H^{1/2,1/4}(\Sigma)\setminus\{0\}}
\frac{\bigl|\langle h_\x^{1/2}(1-\Pi_h)\phi,\,h_\x^{-1/2}(1-\Pi_h)v\rangle_{\Sigma}\bigr|}{\|v\|_{H^{1/2,1/4}(\Sigma)}}\\
&\leq \norm{ h_\x^{1/2}(1-\Pi_h)\phi}{L^2(\Sigma)}
\sup_{v\in H^{1/2,1/4}(\Sigma)\setminus\{0\}}
\frac{\norm{h_\x^{-1/2}(1-\Pi_h)v}{L^2(\Sigma)}}{\|v\|_{H^{1/2,1/4}(\Sigma)}}.
\end{align*}
Since $\Pi_h$ is a local projection, we may replace it by the interpolation operator $\mathcal{I}_h$ and we conclude by Proposition~\ref{lem:interp-stability}  that the second factor is indeed finite.
For uniform parabolically scaled meshes $\PP_h$, standard approximation results show for the first factor that 
$$\norm{ h_\x^{1/2}(1-\Pi_h)\phi}{L^2(\Sigma)} = \OO\big(\norm{h_\x}{L^\infty(\Sigma)}^{\min  \{ p_\x+3/2,2p_t+5/2   \} }\big)= \OO\big(N_h^{-\frac{\min   \{ p_\x+3/2,2p_t+5/2  \} }{d+1}}\big).$$
We conclude that the approximation error introduced by replacing $\phi$ by $\phi_h$  is of higher order.

We next discuss the computaton of $\mathscr{N} \phi_h$ as well as $\langle \mathscr{N} \phi_h,v_h\rangle$. For $(t,\x) \in \Sigma$, we can write  

\begin{align}\label{eq:appB1}
(\mathscr{N} \phi_h)(t,\x) \overset{\eqref{eq:partialnyG}}{=} \sum_{J\times K \in \PP_h} \int_K \int_J \frac{(\y-\x)\cdot\n(\x)}{2(t-s)}G(t-s,\x-\y)\phi_h(s,\y)\d s \d\y.
\end{align}
These evaluations of \eqref{eq:appB1} are required for the numerical quadrature of the error estimator~\eqref{eq:estimator}. 

For piecewise polynomial test functions $v_h$, we can write 

\begin{align}\label{eq:appB2}
&\langle \mathscr{N}\phi_h,v_h\rangle =\\ \notag
& \sum_{J_1\times K_1 \in \PP_h} \sum_{J_2\times K_2 \in \PP_h} \int_{K_1} \int_{K_2} \int_{J_1} \int_{J_2}  \frac{(\y-\x)\cdot\n(\x)}{2(t-s)}G(t-s,\x-\y)\phi_h(s,\y)v_h(t,\x)\d s \d t\d \y \d\x.
\end{align}
These duality products of \eqref{eq:appB2} are required for the numerical quadrature of the Galerkin system~\eqref{eq:Galerkin}.

In Section~\ref{sec:timeintegrals}, we will further simplify the expressions \eqref{eq:appB1}--\eqref{eq:appB2} by analytically computing the time integrals, and in Section~\ref{sec:spaceintegrals}, we will address the numerical computation of the space integrals.

\subsubsection{Initial condition}

To compute the normal derivative of the initial potential applied to the initial condition $U_0$, we introduce a mesh $\mathcal{T}_h$  of $\Omega$. 
For $(t,\x) \in \Sigma$, we can write 
\begin{align}\label{eq:initialcond1}
(\M_1 U_0)(t,\x)&=\sum_{K \in \mathcal{T}_h} \int_K (\y-\x)\cdot\n(\x)\frac{1}{2t} G(t,\x-\y) U_0(\y) \d\y.
\end{align}
Note that the integrands are smooth for $t>0$ so that we can use standard quadrature. 
These evaluations of \eqref{eq:initialcond1} are required for the numerical quadrature of the error estimator~\eqref{eq:estimator}. 

For piecewise polynomial test functions $v_h$, we can write

\begin{align}\label{eq:initialcond2}
\langle \M_1 U_0,v_h \rangle &= \sum_{J_1 \times K_1 \in \PP_h} \sum_{K \in \mathcal{T}_h} \int_{K_1} \int_K \int_{J_1}(\y-\x)\cdot\n(\x)\frac{1}{2t} G(t,\x-\y) U_0(\y) v_h(t,\x) \d t\d\y \d \x.
\end{align}
These duality products of \eqref{eq:initialcond2} are required for the numerical quadrature of the Galerkin system~\eqref{eq:Galerkin}.

In Section~\ref{sec:timeintegrals}, we will further simplify the expression \eqref{eq:initialcond1} by analytically computing the time integral, and in Section~\ref{sec:spaceintegrals}, we will address the numerical computation of the space integrals.

\begin{remark} In our numerical experiments for $d=2$, similarly as in \cite[Remark A.1]{gv21} we construct a triangular mesh $\mathcal{T}_h$ depending on the position of $\x$ for \eqref{eq:initialcond1} and on the position of $K_1$ for \eqref{eq:initialcond2}. In particular, $\mathcal{T}_h$ should read as $\mathcal{T}_h(K_1)$ in \eqref{eq:initialcond2}.
\end{remark}

\subsection{Left-hand side}\label{sec:calctimedepint}

Similarly as in the stationary case \cite[Theorems 6.15 and 6.17]{steinbach08}, for sufficiently smooth functions $u$, $v$ on $\Sigma$, the duality product $\langle \W u,v\rangle_\Sigma$ for the hypersingular integral operator $\W = -\partial_\n \widetilde \K\colon H^{1/2,1/4}(\Sigma)\to H^{-1/2,-1/4}(\Sigma)$ can be rewritten in terms of the single-layer operator $\V:=(\cdot)|_\Sigma \circ \widetilde \V$. 
For sufficiently smooth $\phi$, $\V \phi$ is given as 

	\[(\V\phi)(t,\x):=\int_\Sigma G(t-s,\x-\y)\phi	(s,\y)\d\y\d s \quad \text{for all } (t,\x) \in \Sigma.\]
The following integration-by-parts formula is found in \cite[Theorem 6.1]{costabel90} for $d=2$ and in \cite[ Theorem 2.1]{mst15} for $d=3$; see also \cite{wo22} for a detailed proof. 

\begin{theorem}\label{theorem:hso}
For $u,v\in H^{1,1/2}(\Sigma)$, it holds that
\begin{equation}\label{hsodarstellung}
\langle \mathscr{W}u, v\rangle_\Sigma
= \langle \mathscr{V}\,D_\Gamma u,\; D_\Gamma v\rangle_\Sigma
+ \langle \partial_t\mathscr{V} (u\,\n),\; v\,\n\rangle_\Sigma
\end{equation}
with
$$D_\Gamma (\cdot):=\begin{cases}\nabla_\Gamma (\cdot)\cdot(\n_2, -\n_1)^\top & \quad \text{if } d=2,\\ \n\times \nabla_\Gamma (\cdot) & \quad \text{if } d=3.\end{cases}$$

\end{theorem}
Using another integration by parts \cite[Lemmas 6.13 and 6.16]{steinbach08}, we obtain the following explicit formula for the hypersingular operator
\begin{equation}\label{eq:hypersingularopvar}
\mathscr{W}u=- D_\Gamma'(\mathscr{V} (D_\Gamma u))+\n\cdot \partial_t \mathscr{V}(u \,\n)
\end{equation}
with $$D_\Gamma' (\cdot):=\begin{cases}D_\Gamma (\cdot) & \quad \text{if } d=2,\\ \n\cdot D_\Gamma (\cdot) & \quad \text{if } d=3.\end{cases}$$ 
For a continuous piecewise polynomial $u_h$ with $u_h(0,\cdot) = 0$ and $(t,\x) \in \Sigma$, we can write the first summand in \eqref{eq:hypersingularopvar} without the differential operator $D_\Gamma'$ in front as

\begin{equation}\label{eq:hypersingular1}
	(\V (D_\Gamma u_h))(t,\x) = \sum_{J\times K \in \PP_h} \int_K \int_J G(t-s,\x-\y)D_\Gamma u_h	(s,\y)\d s\d\y.
\end{equation}
For the implementation, we replace the application of the differential operator $D_\Gamma'$ to \eqref{eq:hypersingular1} by the application of $D_\Gamma'$ to the Lagrange interpolation of \eqref{eq:hypersingular1}. 
We can write the  second summand in \eqref{eq:hypersingularopvar} with integration by parts in time applied to $\partial_t G(t-s,\x-\y) = -\partial_s G(t-s,\x-\y) $ as
\begin{align}\label{eq:hypersingular2} \begin{split}
	\n(\x) \cdot(\partial_t \V ( u_h \n))(t,\x) &= \int_\Gamma \int_0^T \partial_t G(t-s,\x-\y)u_h	(s,\y) \n(\x)\cdot \n(\y) \d s\d \y\\ 
	 &= \sum_{J\times K \in \PP_h} \int_K \int_J G(t-s,\x-\y)\partial_s u_h	(s,\y) \n(\x)\cdot \n(\y) \d s\d \y.\end{split}
	\end{align}
The boundary terms vanish since $u_h(0,\cdot)=0$ and $G(t-T,\x-\y)=0$.
The evaluation of \eqref{eq:hypersingularopvar} is required for the computation of the error estimator~\eqref{eq:estimator}.
 
For continuous piecewise polynomial test functions $v_h$, \eqref{eq:hypersingular1}--\eqref{eq:hypersingular2} imply for the two summands in \eqref{hsodarstellung} that

\begin{equation}	\label{eq:hypersingular3}\begin{split}
&\langle\V (D_\Gamma u_h),D_\Gamma v_h\rangle\\ = &\sum_{J_1\times K_1 \in \PP_h} \sum_{J_2\times K_2 \in \PP_h} \int_{K_1} \int_{K_2} \int_{J_1} \int_{J_2} G(t-s,\x-\y)D_\Gamma u_h(s,\y)D_\Gamma v_h	(t,\x) \d s\d t\d \y \d \x, \end{split}
\end{equation}
and

\begin{align}\label{eq:hypersingular4}
&\langle\partial_t \V (u_h \n), v_h \n\rangle\\\notag
 = &\sum_{J_1\times K_1 \in \PP_h} \sum_{J_2\times K_2 \in \PP_h} \int_{K_1} \int_{K_2} \int_{J_1} \int_{J_2} G(t-s,\x-\y)\partial_s u_h	(s,\y) v_h(t,\x)\n(\x)\cdot \n(\y)\d s\d t\d \y \d \x.
\end{align}
The duality products of \eqref{hsodarstellung} are required for the computation of the Galerkin system \eqref{eq:Galerkin}.

In Section~\ref{sec:timeintegrals}, we will further simplify the expressions \eqref{eq:hypersingular1}--\eqref{eq:hypersingular4} by analytically computing the time integrals, and in Section~\ref{sec:spaceintegrals}, we will address the numerical computation of the space integrals.  

\subsection{Analytical computation of time integrals.} \label{sec:timeintegrals}

It is well-known that integration in time for the considered boundary integral operators can be performed exactly in case that the ansatz and test functions are piecewise constant in time; see, e.g., \cite{costabel90,reinarz15,zwom21,gv22}. 
We generalize these explicit formulas to arbitrary polynomial degree in time.
\subsubsection{Single time integrals} As we have seen in the previous sections, the time integrals of $\langle \M_1 U_0,v_h\rangle_\Sigma$ in the Galerkin system~\eqref{eq:Galerkin}  and of the residual in the estimator~\eqref{eq:estimator} can be reduced to
\begin{align}\label{formelintegral0}
&\int_a^b \frac{1}{(t-s)^{1-m}}\,G(t-s,\z)\,s^j\d s
\end{align}
with $m\in\Z$ (more precisely only $m\in\{0,1\}$), a time interval $[a,b]\subseteq \overline{I}$, a time point $t\in I$, $\z\in \R^d\setminus\{0\}$, and $j\in\N_0$.
For $\alpha\in\mathbb{R}$ and $r>0$, we introduce the incomplete gamma function
\begin{align}\label{ableitungGamma}
\Gamma(\alpha,r):=\int_{r}^\infty s^{\alpha-1}e^{-s}~\d s\quad\text{ with derivative }\quad\frac{\d}{\d r}\Gamma(\alpha,r)=-r^{\alpha-1}e^{-r}.
\end{align}
We abbreviate $\delta:=\frac{d}{2}$ and $\zeta:=\frac{|\z|^2}{4}$.
For $t>b$, integration by substitution $(s\mapsto t-\frac{\zeta}{s})$ and the binomial formula show that
\begin{align}\label{eq:innerintegral}
&\int_{a}^{b} \frac{1}{(t-s)^{1-m}}G(t-s,\z)s^j \d s=\frac{1}{(4\pi)^\delta}\int_{a}^{b} \frac{1}{(t-s)^{1-m+\delta}}e^{\frac{-\zeta}{t-s}}s^j \d s\\ \notag
& \quad=\frac{1}{(4\pi)^\delta}\int^{\frac{\zeta}{t-b}}_{\frac{\zeta}{t-a}}   \Bigl( \frac{s}{\zeta}\Bigr)^{1-m+\delta} e^{-s}\frac{\zeta}{s^2}\Bigl(t-\frac{\zeta}{s}\Bigr)^j~\d s\\ \notag
& \quad =\frac{1}{(4\pi)^\delta}\sum_{k=0}^j \binom{j}{k} t^{j-k} (-1)^k\zeta^{m-\delta+k} \int^{\frac{\zeta}{t-b}}_{\frac{\zeta}{t-a}} s^{\delta-k-m-1}e^{-s}~\d s\\ \notag
& \quad =\frac{1}{(4\pi)^\delta}\sum_{k=0}^j\binom{j}{k} t^{j-k} (-1)^k \zeta^{m-\delta+k}\biggl[ \Gamma\Big(\delta-k-m,\frac{\zeta}{t-a}\Big)-\Gamma\Big(\delta-k-m,\frac{\zeta}{t-b}\Big)\biggr].
\end{align}
For general $t\geq 0$, the piecewise definition of $G$ yields that
\begin{equation}\label{formelintegral}
\int_{a}^{b} \frac{1}{(t-s)^{1-m}}G(t-s,\z)s^j \d s=\mathfrak{g}^{j,m}_{t,a}(\z)-\mathfrak{g}^{j,m}_{t,b}(\z)
\end{equation}
with
\begin{align*}
\mathfrak{g}_{t,\sigma}^{j,m}(\z):=\begin{cases}\frac{1}{(4\pi)^\delta}\sum_{k=0}^j \binom{j}{k}(-1)^k \zeta^{m-\delta+k}t^{j-k}  \Gamma(\delta-k-m,\frac{\zeta}{t-\sigma})&\text{if } t-\sigma>0,
\\ 0&\text{else,}\end{cases}
\end{align*}
$\sigma\in \{a,b\}$. 
\subsubsection{Double time integrals}
As we have seen in the previous sections, the time integrals for the duality products in the Galerkin system~\eqref{eq:Galerkin} can be reduced to
\[\int_{a_1}^{b_1}\int_{a_2}^{b_2}\frac{1}{(t-s)^{1-m}}G(t-s,\z)s^jt^i\d s\d t=\int_{a_1}^{b_1}(\mathfrak{g}^{j,m}_{t,a_2}(\z)-\mathfrak{g}^{j,m}_{t,b_2}(\z))t^i\d t\]
with $m\in\Z$ (more precisely only $m\in\{0,1\}$), time intervals $[a_1,b_1],[a_2,b_2]\subseteq \overline{I}$, $\z\in \R^d\setminus\{0\}$, and $i,j\in\N_0$. 
Recalling \eqref{formelintegral}, it remains to calculate the integral of $t^{i+j-k}\Gamma(\delta-k-m,\frac{\zeta}{t-\sigma})$
with respect to $t$. We abbreviate $\beta := \delta -k -m$. 
Then, the substitution $t \mapsto t+\sigma$  gives that $$\int_{a_1}^{b_1}   t^{i+j-k}\Gamma\Big(\beta,\frac{\zeta}{t-\sigma}\Big)\d t  =\int_{a_1-\sigma}^{b_1-\sigma} \sum_{\ell=0}^{i+j-k}\binom{i+j-k}{\ell}\sigma^{i+j-k-\ell}t^\ell\Gamma\Big(\beta,\frac{\zeta}{t}\Big)\d t$$
 For $t \in \{a_1,b_1\}$, integration by parts (using \eqref{ableitungGamma} to get $\frac{\d}{\d t}\Gamma(\beta,\frac{\zeta}{t})=\zeta^\beta t^{-1-\beta} e^{-\zeta/t}$) and integration by substitution $(t\mapsto \frac{\zeta}{t})$ give that
\begin{align*}
\int_{a_1-\sigma}^{b_1-\sigma} t^\ell \Gamma \Big(\beta,\frac{\zeta}{t}\Big)\d t&=\biggl[\frac{t^{\ell+1}}{\ell+1}\Gamma\Big(\beta,\frac{\zeta}{t}\Big) \biggr]_{a_1-\sigma}^{b_1-\sigma}-\int_{a_1-\sigma}^{b_1-\sigma} \frac{t^{\ell+1}}{\ell+1}\zeta^\beta t^{-1-\beta}e^{-\zeta/t}\d t\\
&=\biggl[\frac{t^{\ell+1}}{\ell+1}\Gamma\Big(\beta,\frac{\zeta}{t}\Big) \biggr]_{a_1-\sigma}^{b_1-\sigma}-\frac{1}{\ell+1}\int_{\zeta/(a_1-\sigma)}^{\zeta/(b_1-\sigma)} \Big(\frac{\zeta}{t}\Big)^{\ell-\beta}\zeta^\beta \frac{-\zeta}{t^2}e^{-t}\d t\\
&=\frac{1}{\ell+1}\Bigl[t^{\ell+1}\Gamma\Big(\beta,\frac{\zeta}{t}\Big)-\zeta^{\ell+1}\Gamma\Big(\beta-\ell-1,\frac{\zeta}{t}\Big)\Bigr]_{a_1-\sigma}^{b_1-\sigma}=:\mathfrak{H}^{\ell,\beta}_{b_1-\sigma}(\z)-\mathfrak{H}^{\ell,\beta}_{a_1-\sigma}(\z).
\end{align*} 

The piecewise definition of $\mathfrak{g}_{t,\sigma}^{j,m}$ yields that
\begin{equation}\label{formeldoppelintegral}
\int_{a_1}^{b_1}\int_{a_2}^{b_2} \frac{1}{(t-s)^{1-m}}G(t-s,\z)s^jt^i \d s\d t=\mathfrak{G}^{i,j,m}_{a_1,a_2}(\z)-\mathfrak{G}^{i,j,m}_{b_1,a_2}(\z)-\mathfrak{G}^{i,j,m}_{a_1,b_2}(\z)+\mathfrak{G}^{i,j,m}_{b_1,b_2}(\z),
\end{equation} with
\begin{align*}
\mathfrak{G}^{i,j,m}_{\tau,\sigma}(\z):=\begin{cases}
\frac{1}{(4\pi)^\delta}\sum_{k=0}^j \binom{j}{k}(-1)^k\zeta^{m-\delta+k} & \\ \quad\quad \sum_{\ell=0}^{i+j-k}\binom{i+j-k}{\ell}\sigma^{i+j-k-\ell}\mathfrak{H}^{\ell,\delta-k-m}_{\tau-\sigma}(\z)&\text{if } \tau-\sigma>0,\\
0&\text{else},
\end{cases}
\end{align*}
$\tau \in \{a_1,b_1\}$, $\sigma\in\{a_2,b_2\}$.

\subsection{Approximate computation of space integrals}\label{sec:spaceintegrals}

It remains to discuss the numerical quadrature of the singular integrals 

\begin{equation}\label{Einfachintegral}
\int_{K_2} [(\y - \x)\cdot\n(\x)]^{1-m} \mathfrak{g}_{t,\sigma}^{j,m}(\x-\y) u_{h_\x}(\y) \d\y
\end{equation}
and    
\begin{equation}\label{Doppelintegral}
\int_{K_1} \int_{K_2} [(\y - \x)\cdot \n(\x)]^{1-m}  \mathfrak{G}_{\tau,\sigma}^{i,j,m}(\x-\y) u_{h_\x}(\y) v_{h_\x}(\x) \d\y \d\x
\end{equation}
for $(t,\x)\in \Sigma$, $J_1 \times K_1, J_2 \times K_2 \in \PP_h$, $\tau,\sigma\in \overline{I}$ with $\tau>\sigma$ , $i, j \in \N_0$, $m \in \{0,1\}$, and piecewise polynomials $u_{h_\x}$ and $v_{h_\x}$ on $\Gamma$. 
For the term $\langle \M_1 U_0,v_h\rangle_\Sigma$, we additionally need to consider
\begin{equation}\label{Doppelintegral2}
\int_{K_1}\int_K [(\y - \x)\cdot \n(\x)]^{1-m} \mathfrak{g}_{t,\sigma}^{i,m}(\x-\y) U_0(\y) v_{h_\x}(\x) \d\y \d\x
\end{equation}
for $K \subset \Omega$. 

Since the boundary $\Gamma$ is piecewise smooth, it holds that $(\y-\x)\cdot \n(\x)=\OO(|\x-\y|^2)$ for $\x,\y$ on the same smooth part of $\Gamma$ and $|\x-\y|\to 0$; see, e.g., \cite[Lemma 2.2.14]{ss11}. (In fact, for piecewise flat $\Gamma$, we even have that $(\y-\x)\cdot \n(\x) = 0$.)
It remains to determine the strength of the singularity of $\mathfrak{g}^{j,m}_{t,\sigma}(\mathbf{z})$ and $\mathfrak{G}^{\,i,j,m}_{\tau,\sigma}(\mathbf{z})$ at $\z=0$.
For $\mathfrak{g}^{j,m}_{t,\sigma}(\mathbf{z})$ from \eqref{formelintegral}, the only term that requires a detailed study
is
\[
\zeta^{\,m-\delta+k}\,
\Gamma\Bigl(\delta-k-m,\frac{\zeta}{t-\sigma}\Bigr)
\]
for $k \in \{0,\ldots,j\}$.
Here, we use again the abbreviations $\delta=d/2$ and $\zeta=|\mathbf{z}|^2/4$.
Similarly, for 
$\mathfrak{G}^{\,i,j,m}_{\tau,\sigma}(\mathbf{z})$ from \eqref{formeldoppelintegral}, it is sufficient to study the term
\[\zeta^{m-\delta+k}\Bigl(\zeta^{\ell+1}\Gamma(\delta-k-m-\ell-1,\tfrac{\zeta}{\tau-\sigma})
- (\tau-\sigma)^{\ell+1}\Gamma(\delta-k-m,\tfrac{\zeta}{\tau-\sigma})\Bigr)
\]
for $k \in \{0,\ldots,j\}$ and $\ell \in \{0,\ldots,i+j-k\}$. Abbreviating $\xi:=\zeta/(\tau - \sigma)$, all these terms are of the form 
\begin{equation}\label{eq:reduced-term}
\xi^{n-\delta}\,\Gamma(\delta-n,\xi)
\end{equation}
for some $n \in \N_0$ (either $n = k+m$ or $n = k+m+\ell+1$). The following lemma characterizes the singularity of this term. Here, $\Gamma(\cdot)$ denotes the (complete) gamma function.
\begin{lemma}\label{lem:gamma-singularity}
Let $\delta=\tfrac d2$ and $n\in\mathbb{N}_0$. For $\xi \searrow 0$, it holds that

\[\xi^{n-\delta}\,\Gamma(\delta-n,\xi)=\begin{cases}
\mathcal{O}(\xi^{-1}) &\text{if } d=2 \text{ and } n=0,\\
\mathcal{O}(\xi^{n-1} \ln (\xi))&\text{if } d=2 \text{ and } n\geq1,\\
\mathcal{O}(\xi^{n-3/2}) &\text{if } d=3 \text{ and } n\in \{0,1\},\\
\mathcal{O}(\xi^{n-2})&\text{if } d 	=3 \text{ and } n\ge 2.
\end{cases}\]
\end{lemma}

\begin{proof}
We treat the cases $n<\delta$, $n=\delta$, and $n>\delta$ separately.

\textbf{Case $n<\delta$:} The incomplete gamma function has the following power series representation
\[ \xi^{n-\delta}\Gamma(\delta-n,\xi)= \xi^{n-\delta}\Gamma(\delta-n)-\Gamma(\delta-n)e^{-\xi}\sum_{k=0}^\infty\frac{\xi^k}{\Gamma(\delta-n+1+k)}.\]

\textbf{Case $n=\delta$:} 
In this case,  $\Gamma(0,\xi)=-\operatorname{Ei}(-\xi)$,
and the exponential integral has the following power series representation $-\operatorname{Ei}(-\xi) = -\gamma - \ln(\xi) - \sum_{k=1}^\infty \frac{(-\xi)^k}{k\cdot k!}$.

\textbf{Case $n>\delta$:} Integration by parts shows that
\[\Gamma(\delta-(n-1),\xi)=\int_{\xi}^\infty y^{\delta-n}e^{-y}\d y=\xi^{\delta-n}e^{-\zeta}+(\delta-n)\Gamma(\delta-n,\xi).\]
Thus, we obtain the recursion
\begin{equation}
\xi^{n-\delta}\Gamma(\delta-n,\xi)=\frac{1}{\delta-n}\Bigl(\xi\cdot \xi^{(n-1)-\delta}\Gamma(\delta-(n-1),\xi)-e^{-\xi}\Bigr).
\end{equation}
Employing this recursion $n-1$ times, we conclude the proof using the first and the second case.
\end{proof} 
We conclude that the dominant singularity in our case occurs for $n = m$. 
If $m=1$ (which occurs for the operator $\W$, see \eqref{eq:hypersingularopvar}--\eqref{eq:hypersingular4}), the terms $\mathfrak{g}_{t,\sigma}^{j,m}(\x-\y)$ and $\mathfrak{G}_{\tau,\sigma}^{i,j,m}(\x-\y)$ have the singularity $\ln|\x-\y|$ for $d=2$ and $|\x-\y|^{-1}$ for $d=3$. 
If $m=0$ (which occurs for the operator $\mathscr{N}$ and $\M_1$, see \eqref{eq:appB1}, \eqref{eq:appB2}, and \eqref{eq:initialcond2}), and $\x,\y$ lie on the same smooth part of $\Gamma$, the terms $(\y-\x)\cdot \n(\x)  \mathfrak{g}_{t,\sigma}^{j,m}(\x-\y)$ and $(\y-\x)\cdot\n(\x) \mathfrak{G}_{\tau,\sigma}^{i,j,m}(\x-\y)$ are regular for $d=2$ and have the singularity $|\x-\y|^{-1}$ for $d=3$, since $(\y-\x)\cdot \n(\x)=\OO(|\x-\y|^2)$ for $\x \to \y$.

The integrals \eqref{Einfachintegral}--\eqref{Doppelintegral2} can now be computed as in  \cite[Appendix A]{gv22} (in case of a triangular $K \subset \Omega$), using Duffy transformations as in  \cite[Section 5.1]{gps22b} together with the quadrature rules from \cite{smith00} for logarithmic singularities on intervals.

\begin{remark}
In our numerical experiments for $d=2$, the reduced term \eqref{eq:reduced-term} must be evaluated very frequently so that a fast evaluation becomes crucial.  If $m = 0$, the term is given as $\xi^{-1} \Gamma(1,\xi) = \xi^{-1}e^{-\xi}$.
In the case of $n \ge 1$, the arguments of the proof of Lemma~\ref{lem:gamma-singularity} yield the expression
\[\xi^{n-1} \Gamma(1-n,\xi)= -\frac{(-\xi)^{n-1}}{(n-1)!}{\rm Ei}(-\xi)+\frac{e^{-\xi}}{(n-1)!}\sum_{k=0}^{n-2}(n-k-2)!(-\xi)^k. \]
 The evaluation of the exponential integral ${\rm Ei}$ is very time-consuming, so we use the following approximations instead:
\begin{itemize}
\item For small $\xi$, we can use the well-known convergent series representation
\begin{align*}
{\rm Ei}(-\xi) &= \gamma + {\rm Re}(\ln(-\xi))-e^{-\xi/2}\sum_{k=1}^\infty\frac{\xi^k}{k!2^{k-1}}\sum_{l=0}^{ \lfloor(k-1)/2 \rfloor}\frac{1}{2l+1}.
\end{align*}
The series is truncated depending on the value of $\xi$.
\item For large $\xi$, an increasing number of terms in this series must be calculated to obtain a good approximation, because the term ${\xi}^k / (k!2^{k-1})$ now decreases more slowly as k increases. In this case, we thus use the approximation described in \cite{bpl00} instead.
\end{itemize}
\end{remark}

\bibliographystyle{alpha}
\bibliography{literature_sthyp}

\end{document}